\documentclass{siamart1116}

\usepackage{tikz}
\usepackage{color}
\usepackage{pgfplots}
\usepackage{pgfplotstable}
\usepackage{booktabs}
\usepackage{listings}
\usepackage{url}
\usepackage{amsmath,amssymb,amsfonts,cite}
\usetikzlibrary{patterns}
\usepackage{enumerate}
\usepackage{algorithm}
\usepackage{algpseudocode}
\usepackage{footnote}
\makesavenoteenv{table}
\makesavenoteenv{tabular}
\renewcommand{\tilde}{\widetilde}
\newcommand{\hm}{hm-toolbox}
\newcommand{\matlab}{Matlab}
\colorlet{DenseBlockColor}{gray!60}
\newcommand{\mytexttt}[1]{\color{stringcolor}\texttt{#1}}

\pgfplotstableset{
	every head row/.style={before row=\toprule,after row=\midrule},
	clear infinite
}

\DeclareMathOperator{\diag}{diag}
\newcommand{\norm}[1]{\lVert#1\rVert}
\definecolor{mygreen}{RGB}{28,172,0} % color values Red, Green, Blue
\definecolor{mylilas}{RGB}{170,55,241}
\definecolor{stringcolor}{RGB}{180,10,10}
\definecolor{mygray}{RGB}{240,240,240}
\definecolor{mygray2}{RGB}{200,200,200}
\author{
	Stefano Massei\thanks{EPF Lausanne, Switzerland,
		\email{stefano.massei@epfl.ch}. The work of Stefano Massei has been supported by the SNSF research project \emph{Fast algorithms from low-rank updates}, grant number: 200020\_178806.} \and
	Leonardo Robol\thanks{Department of Mathematics, University of Pisa, 
		\email{leonardo.robol@unipi.it}. The work of Leonardo Robol was partially supported by the GNCS/INdAM project ``Metodi di proiezione per equazioni di 
		matrici e sistemi lineari con operatori definiti tramite somme di 
		prodotti di Kronecker, e soluzioni con struttura di rango''.} \and 	Daniel Kressner\thanks{EPF Lausanne, Switzerland,
			\email{daniel.kressner@epfl.ch}}
}

\title{\lowercase{hm-toolbox}: Matlab software \\ for HODLR and HSS matrices}

\begin{document}
\maketitle
\begin{abstract}
Matrices with hierarchical low-rank structure, including HODLR and HSS matrices, constitute a versatile tool to develop fast algorithms for addressing large-scale problems. While existing software packages for such matrices often focus on linear systems, their scope of applications is in fact much wider and includes, for example, matrix functions and eigenvalue problems. In this work, we present a new \matlab{} toolbox called \hm, which encompasses this versatility with a broad set of tools for HODLR and HSS matrices, unmatched by existing software. While mostly based on algorithms that can be found in the literature, our toolbox also contains a few new algorithms as well as novel auxiliary functions. Being entirely based on \matlab, our implementation  does not strive for optimal performance. Nevertheless, it  maintains the favorable complexity of hierarchical low-rank matrices and offers, at the same time, a convenient way of prototyping and experimenting with algorithms.  A number of applications illustrate the use of the \hm.
\end{abstract}
\bigskip


\bigskip 


\section{Introduction}

This work presents \hm, a new \matlab{} software available from \url{https://github.com/numpi/hm-toolbox} for working with HODLR (hierarchically off-diagonal low-rank) and HSS (hierarchically semi-separable) matrices. Both formats are defined via a recursive block partition of the matrix. More specifically, 
for \begin{equation} \label{eq:hodler1levelxx}
	A = \begin{bmatrix}
		A_{11}&A_{12}\\
		A_{21}&A_{22}
	\end{bmatrix},
\end{equation}
it is assumed that the off-diagonal blocks $A_{12}$, $A_{21}$ have low rank. This partition is repeated recursively for the diagonal blocks until a minimal block size is reached. In the HSS format, the low-rank factors representing the off-diagonal blocks on the different levels of the recursions are nested, while the HODLR format treats all off-diagonal blocks independently. During the last decade, both formats have shown their usefulness in a wide variety of applications. Recent examples include 
the acceleration of sparse direct linear system solvers~\cite{Ghysels2016,Wang2016,Xia2009}, large-scale Gaussian process modeling~\cite{Ambikasaran2016,Geoga2019}, stationary distribution of quasi-birth-death Markov chains \cite{Bini2017cyclic},
as well as fast solvers for (banded) eigenvalue problems~\cite{KressnerS2017,SusnjaraK2018,Vogel2016} and matrix equations~\cite{kressner2017low,KressnerKM2019}. 

Both, the HODLR and the HSS formats, allow to design fast algorithms for various linear algebra tasks. Our toolbox offers basic operations (addition, multiplication, inversion), matrix decompositions (Cholesky, LU, QR, ULV), as well as more advanced functionality (matrix functions, solution of matrix equations). It also offers multiple ways of constructing and recompressing these representations as well as converting between HODLR, HSS, and sparse matrices. While most of the toolbox is based on known algorithms from the literature, we also make novel algorithmic contributions. This includes the fast computation of Hadamard products, the matrix product $A^{-1} B$ for HSS matrices $A,B$, and numerous auxiliary functionality.

The primary goal of the \hm{} is to provide a comprehensive and convenient framework for prototyping algorithms and ensuring reproducibility. Having this goal in mind, our implementation is entirely based on \matlab{} and thus does not strive for optimal performance. Still, the favorable complexity of the fast algorithms is preserved.

The HODLR and HSS formats are special cases of hierarchical and $\mathcal H^2$ matrices, respectively. The latter two formats allow for significantly more general block partitions, described via cluster trees, which in turn gives the ability to treat a wider range of problems effectively, including two- and three-dimensional partial differential equations; see~\cite{Hackbusch2015} and the references therein. On the other hand, the restriction to partitions of the form~\eqref{eq:hodler1levelxx} comes with a major advantage; it simplifies the design, analysis, and implementation of fast algorithms. Another advantage of~\eqref{eq:hodler1levelxx} is that a low-rank perturbation makes $A$ block diagonal, which opens the door for divide-and-conquer methods; see~\cite{kressner2017low} for an example. 

\begin{paragraph}{Existing software}
	In the following, we provide a brief overview of existing software for various flavors of hierarchical low-rank formats.
	An $n\times n$ matrix $S$ is called \emph{semiseparable} if every submatrix residing entirely in the upper or lower triangular part of $S$ has rank at most one. The class of \emph{quasiseparable} matrices is more general by only considering submatrices in the \emph{strictly} lower and upper triangular parts. The class of sequentially semiseparable matrices is another
	generalization, which has been defined in~\cite{Chandrasekaran2005}.

	While a fairly complete \matlab{} library for semiseparable matrices is available\footnote{\url{https://people.cs.kuleuven.be/~raf.vandebril/homepage/software/sspack.php}}, the public availability of software for quasiseparable matrices seems to be limited to a set of \matlab{} functions targeting specific tasks\footnote{\url{http://people.cs.dm.unipi.it/boito/software.html}}.

	Fortran and \matlab{} packages for solving linear systems with HSS and sequentially semiseparable matrices are available\footnote{\url{http://scg.ece.ucsb.edu/software.html}}\footnote{\url{http://www.math.purdue.edu/~xiaj/packages.html}}. The Structured Matrix Market\footnote{\url{http://smart.math.purdue.edu/}} provides benchmark examples and supporting functionality for HSS matrices.
	STRUMPACK~\cite{Rouet2016} is a parallel C++ library for HSS matrices with a focus on randomized compression and the solution of linear systems.
	HODLRlib~\cite{Ambikasaran2019HODLRlib} is a C++ library for HODLR matrices, which provides shared-memory parallelism through OpenMP and 
	again puts a focus on linear systems. HLib \cite{Borm2003} and H2Lib \cite{H2Lib}  are C libraries which provide a wide range of functionality for hierarchical and $\mathcal H^2$ matrices, respectively. 
	HLIBpro \cite{HLIBpro} and AHMED \cite{AHMED}
	are  C++ libraries implementing optimized algorithms
	for $\mathcal H$-matrices. Pointers to other software packages, related to hierarchical low-rank formats, can be found at \url{https://github.com/gchavez2/awesome_hierarchical_matrices}. 
	
\end{paragraph}

\begin{paragraph}{Outline}
	The rest of this work is organized as follows. In Section~\ref{sec:prelims}, we recall the definitions of HODLR and HSS matrices. Section~\ref{sec:construction} is concerned with the construction of such matrices in our toolbox and the conversion between different formats. In Section~\ref{sec:arithmetic}, we give a brief overview of those arithmetic operations implemented in the \hm{} that are based on existing algorithms. More details are provided on two new algorithms and the important recompression operation. Finally, in Section~\ref{sec:examples}, we illustrate the use of our toolbox with various examples and applications.
\end{paragraph}

\section{Preliminaries and Matlab classes \lstinline|hodlr|, \lstinline|hss|} \label{sec:prelims}
\subsection{HODLR matrices}

As discussed in the introduction, HODLR matrices are defined via a recursive block partition~\eqref{eq:hodler1levelxx}, assuming that the off-diagonal blocks have low rank on every level of the recursion.

The concept of a \emph{cluster tree} allows to formalize the definition of such a partitioning.

\begin{definition} \label{def:clustertree}
	Given $n\in \mathbb N$, let $\mathcal T_p$ be a completely balanced binary tree of depth $p$ whose nodes are
	subsets of $\{ 1, \ldots, n \}$. We say that $\mathcal T_p$ is
	a \emph{cluster tree} if it satisfies:
	\begin{itemize}
		\item The root is $I^{0}_1 := I = \{ 1, \ldots, n \}$.
		\item The nodes at level $\ell$, denoted by $I^{\ell}_1, \ldots, I^{\ell}_{2^\ell}$, form
		a partitioning of $\{1, \ldots, n\}$ into consecutive indices:
		\[
		I_i^\ell = \{ n^{(\ell)}_{i-1}+1\ldots, n^{(\ell)}_{i}-1, n_i^{(\ell)} \}
		\]
		for some integers $0 = n^{(\ell)}_{0} \le  n^{(\ell)}_{1} \le \cdots \le n^{(\ell)}_{2^\ell} = n$, $\ell = 0,\ldots p$. In particular, if $n_{i-1}^{(\ell)}=n_i^{(\ell)}$ then $I_i^{\ell}=\emptyset$.
		\item The node $I_{i}^{\ell}$ has children 
		$I_{2i-1}^{\ell+1}$ and $I_{2i}^{\ell+1}$, for any $1 \leq \ell \leq p - 1$. The children form a partitioning of their parent. 
	\end{itemize}
\end{definition}

In practice, the cluster tree $\mathcal T_p$ is often chosen in a balanced fashion, that is, the cardinalities of the index sets on the same level are nearly equal and the depth of the tree is determined by a minimal diagonal block
size $n_{\mathrm{min}}$ for stopping the recursion.

In particular, if $n = 2^p n_{\mathrm{min}}$, such a construction yields a
perfectly balanced binary tree of depth $p$, see Figure~\ref{fig:blockpart} for $n = 8$
and $n_{\mathrm{min}} = 1$.  

\begin{figure}
	\centering 
	\begin{tikzpicture}[scale=0.8] \small
	\coordinate (T18) at (0,0);  
	\coordinate (T14) at (-4,-1);
	\coordinate (T58) at (4,-1);
	\coordinate (T12) at (-6,-2);
	\coordinate (T34) at (-2,-2);
	\coordinate (T56) at (2,-2);
	\coordinate (T78) at (6,-2);
	\coordinate (T1) at (-7,-3);
	\coordinate (T2) at (-5,-3);
	\coordinate (T3) at (-3,-3);
	\coordinate (T4) at (-1,-3);
	\coordinate (T5) at (1,-3);
	\coordinate (T6) at (3,-3);
	\coordinate (T7) at (5,-3);
	\coordinate (T8) at (7,-3);
	\node (N18) at (T18) {$I = \{1,2,3,4,5,6,7,8\}$};
	\node (N14) at (T14) {$I_1^1 = \{1,2,3,4\}$};
	\node (N58) at (T58) {$I_2^1 = \{5,6,7,8\}$};
	\node (N12) at (T12) {$I_1^2 = \{1,2\}$};
	\node (N34) at (T34) {$I_2^2 = \{3,4\}$};
	\node (N56) at (T56) {$I_3^2 = \{5,6\}$};
	\node (N78) at (T78) {$I_4^2 = \{7,8\}$};
	\node (N1) at (T1) {$I_1^3 = \{1\}$};
	\node (N2) at (T2) {$I_2^3 = \{2\}$};
	\node (N3) at (T3) {$I_3^3 = \{3\}$};
	\node (N4) at (T4) {$I_4^3 = \{4\}$};
	\node (N5) at (T5) {$I_5^3 = \{5\}$};
	\node (N6) at (T6) {$I_6^3 = \{6\}$};   
	\node (N7) at (T7) {$I_7^3 = \{7\}$};
	\node (N8) at (T8) {$I_8^3 = \{8\}$};
	\draw[->] (N18.south) -- (N14);
	\draw[->] (N18.south) -- (N58);  
	\draw[->] (N14.south) -- (N12);
	\draw[->] (N14.south) -- (N34);    
	\draw[->] (N58.south) -- (N56);  
	\draw[->] (N58.south) -- (N78);
	\draw[->] (N12.south) -- (N1);
	\draw[->] (N12.south) -- (N2);
	\draw[->] (N34.south) -- (N3);  
	\draw[->] (N34.south) -- (N4);  
	\draw[->] (N56.south) -- (N5);  
	\draw[->] (N56.south) -- (N6);          
	\draw[->] (N78.south) -- (N7);  
	\draw[->] (N78.south) -- (N8);          
	\end{tikzpicture}

	\begin{tikzpicture}[scale=0.3]    
	\begin{scope}
	[xshift=0cm]
	\node [above] at (4,8) {$\ell=0$};
	\draw (0,0) -- (0,8) -- (8,8) -- (8,0) -- cycle;
	\end{scope}
	\begin{scope}
	[xshift=10cm]
	\node [above] at (4,8) {$\ell=1$};
	\draw (0,0) -- (0,8) -- (8,8) -- (8,0) -- cycle;
	\draw[pattern=north west lines, pattern color=blue] (0,0) rectangle (4,4);
	\draw[pattern=north west lines, pattern color=blue] (4,4) rectangle (8,8);
	\draw (0,4) -- (8,4);
	\draw (4,0) -- (4,8);
	\end{scope}
	\begin{scope}
	[xshift=20cm]
	\node [above] at (4,8) {$\ell=2$};
	\draw (0,0) -- (0,8) -- (8,8) -- (8,0) -- cycle;
	\draw[pattern=north west lines, pattern color=blue] (4,0) rectangle (6,2);
	\draw[pattern=north west lines, pattern color=blue] (0,4) rectangle (2,6);
	\draw[pattern=north west lines, pattern color=blue] (6,2) rectangle (8,4);
	\draw[pattern=north west lines, pattern color=blue] (2,6) rectangle (4,8);
	\draw (0,2) -- (8,2);
	\draw (0,4) -- (8,4);
	\draw (0,6) -- (8,6);
	\draw (2,0) -- (2,8);
	\draw (4,0) -- (4,8);
	\draw (6,0) -- (6,8);
	\end{scope}
	\begin{scope}
	[xshift=30cm]
	\node [above] at (4,8) {$\ell=3$};
	\draw (0,0) -- (0,8) -- (8,8) -- (8,0) -- cycle;
	\draw[pattern=north west lines, pattern color=blue] (6,0) rectangle (7,1);
	\draw[pattern=north west lines, pattern color=blue] (4,2) rectangle (5,3);
	\draw[pattern=north west lines, pattern color=blue] (2,4) rectangle (3,5);
	\draw[pattern=north west lines, pattern color=blue] (0,6) rectangle (1,7);
	\draw[pattern=north west lines, pattern color=blue] (7,1) rectangle (8,2);
	\draw[pattern=north west lines, pattern color=blue] (5,3) rectangle (6,4);
	\draw[pattern=north west lines, pattern color=blue] (3,5) rectangle (4,6);
	\draw[pattern=north west lines, pattern color=blue] (1,7) rectangle (2,8);
	
	\draw (0,1) -- (8,1);
	\draw (0,2) -- (8,2);
	\draw (0,3) -- (8,3);
	\draw (0,4) -- (8,4);
	\draw (0,5) -- (8,5);
	\draw (0,6) -- (8,6);
	\draw (0,7) -- (8,7);
	\draw (1,0) -- (1,8);
	\draw (2,0) -- (2,8);
	\draw (3,0) -- (3,8);
	\draw (4,0) -- (4,8);
	\draw (5,0) -- (5,8);
	\draw (6,0) -- (6,8);
	\draw (7,0) -- (7,8);
	\end{scope}
	\end{tikzpicture}
	
	\caption{Pictures taken from \cite{kressner2017low}: Example of a cluster tree of depth $3$ and the block partitions induced on each level. The blocks marked with blue stripes are stored as low-rank matrices in the HODLR format.}\label{fig:blockpart}
\end{figure}
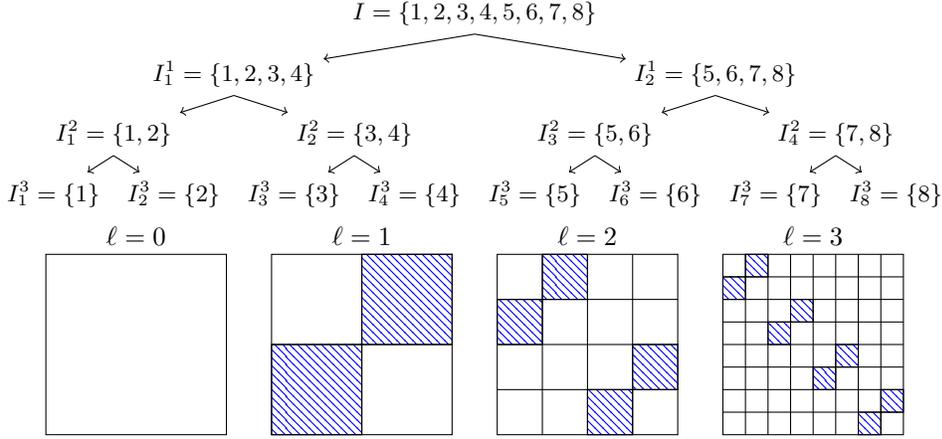

The nodes at a level $\ell$ induce a partitioning of $A$ into a $2^\ell \times 2^\ell$ block
matrix, with the blocks given by $A(I^\ell_i, I^\ell_j)$ for $i,j = 1,\ldots,2^\ell$, where we use Matlab notation for submatrices.

The definition of a HODLR matrix requires that some of the off-diagonal blocks (marked with stripes in Figure~\ref{fig:blockpart}) have (low) bounded rank.

\begin{definition} \label{def:hodlr}
	Let $A\in \mathbb C^{n\times n}$ and consider a cluster tree $\mathcal T_p$. 
	\begin{enumerate}
		\item Given $k\in \mathbb N$, $A$ is said to be a $(\mathcal T_p,k)$-HODLR matrix if every off-diagonal block
		\begin{equation}\label{eq:off-diag-block}
		A(I_i^\ell,I_j^\ell) \quad \text{such that $I_i^\ell$ and $I_j^\ell$ are siblings in $\mathcal T_p$}, \quad \ell=1,\dots,p,
		\end{equation} has rank at most $k$.
		\item The HODLR rank of $A$ (with respect to $\mathcal T_p$)  is the smallest integer $k$ such that $A$ is a $(\mathcal T_p,k)$-HODLR matrix.
	\end{enumerate}
\end{definition}

\begin{paragraph}{Matlab class} The \hm{} provides the Matlab class \lstinline|hodlr| for working with HODLR matrices. The properties of \lstinline|hodlr| store a matrix recursively in accordance with the partitioning~\eqref{eq:hodler1levelxx} (or, equivalently, the cluster tree) as follows:
	\begin{itemize}
		\item \lstinline|A11| and \lstinline|A22| are \lstinline|hodlr| instances representing the diagonal blocks (for a nonleaf node); 
		\item \lstinline|U12| and \lstinline|V12| are the low-rank factors of the off-diagonal block $A_{12}$;
		\item \lstinline|U21| and \lstinline|V21| are the low-rank factors of the off-diagonal block $A_{21}$;
		\item \lstinline|F| is either a dense matrix representing the whole matrix (for a leaf node) or empty.
	\end{itemize}

	Figure~\ref{fig:hodlr} illustrates the storage format. For a matrix of HODLR rank $k$, the memory consumption reduces from $\mathcal O(n^2)$ to $\mathcal O(pnk) = \mathcal O(kn\log n)$ when using \lstinline|hodlr|.
\end{paragraph}

\begin{figure}[!ht]
	\centering
	\includegraphics[width=0.8\textwidth]{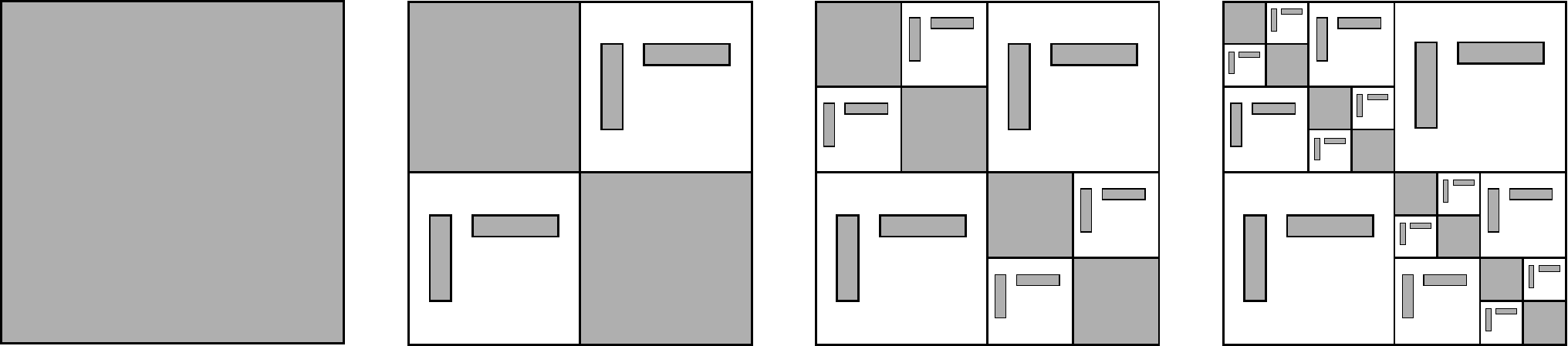}
	\caption{ Image taken from \cite{kressner2017low}: Illustration of the HODLR format for cluster trees of varying depth. The gray blocks are the (dense) matrices that need to be stored to represent a HODLR matrix.}\label{fig:hodlr}
\end{figure}

\subsection{HSS matrices}

The $\log(n)$ factor in the memory complexity of HODLR matrices arises from the fact that the low-rank factors take $\mathcal O(kn)$ memory on each of the $\mathcal O(\log(n))$ levels of the recursion. However, in many -- if not most -- applications these factors share similarities across different levels, which can be exploited by nested hierarchical low-rank formats, such as the HSS format, to potentially remove the $\log(n)$ factor.

An HSS matrix is associated with a cluster tree $\mathcal T_p$; see Definition~\ref{def:clustertree}.
In analogy to HODLR matrices, it is assumed that the off-diagonal blocks can be factorized as 
\[
A(I^\ell_i,I^\ell_j) = U_i^{(\ell)} S_{i,j}^{(\ell)} (V_j^{(\ell)})^*, \quad S_{i,j}^{(\ell)} \in\mathbb C^{k\times k}, \quad U_i^{(\ell)}\in \mathbb C^{n_i^{(\ell)} \times k},\quad V_j^{(\ell)} \in \mathbb C^{n_j^{(\ell)} \times k},
\]
for all siblings $I^\ell_i,I^\ell_j$ in $\mathcal T_p$. The matrices $S_{i,j}^{(\ell)}$ are called \emph{core blocks}.
Additionally, and in contrast to HODLR matrices, for HSS matrices we require the factors $U_i^{(\ell)}, V_j^{(\ell)}$ to be nested across different levels of $\mathcal T_p$. More specifically, it is assumed that there exist so called \emph{translation operators}, $R_{U,i}^{(\ell)}, R_{V,j}^{(\ell)}\in\mathbb C^{2k\times k}$ such that
\begin{equation}\label{eq:HSS_nestedness}
U_i^{(\ell)} = \begin{bmatrix} U_{i_1}^{(\ell+1)} & 0 \\ 0 & U_{i_2}^{(\ell+1)} \end{bmatrix} R_{U,i}^{(\ell)}, \qquad
V_j^{(\ell)} = \begin{bmatrix} V_{j_1}^{(\ell+1)} & 0 \\ 0 & V_{j_2}^{(\ell+1)} \end{bmatrix} R_{V,j}^{(\ell)}, 
\end{equation}
where $I_{i_1}^{\ell+1}, I_{i_2}^{\ell+1}$  and $I_{j_1}^{\ell+1}, I_{j_2}^{\ell+1}$ denote the children of $I_{i}^{\ell}$  and $I_{j}^{\ell}$, respectively. These relations
allow to retrieve the low-rank factors $U_i^{(\ell)}$ and $V_i^{(\ell)}$ for the higher levels $\ell=1,\ldots,p-1$ recursively from the bases $U_i^{(p)}$ and $V_i^{(p)}$ at the deepest level $p$. Therefore, in order to represent $A$, one only needs to store: the diagonal blocks $D_i:=A(I^p_i,I^p_i)$, the bases $U_i^{(p)}$, $V_i^{(p)}$, the core factors $S_{i,j}^{(\ell)}$, $S_{j,i}^{(\ell)}$ and the translation operators $R_{U,i}^{(\ell)}$, $R_{V,i}^{(\ell)}$. In particular,
note that  only the bases on the lowest level, $U_i^{(p)}$ and 
$V_i^{(p)}$, are stored. 
We remark that, for simplifying the exposition, we have considered translation operators and bases $U_i^{(p)},V_j^{(p)}$ with $k$ columns for every level and node. This is not necessary, as long as the dimensions are compatible, and this more general framework is handled in \hm.

As explained in~\cite{Hackbusch2004}, a matrix $A$ admits the decomposition explained above if and only if it is an HSS matrix in the sense of the following definition, which imposes rank conditions on certain block rows and columns without their diagonal blocks.
\begin{definition} \label{def:hss}
	Let $A\in \mathbb C^{n\times n}$, $I = \{1,\ldots, n\}$, and consider a cluster tree $\mathcal T_p$. 
	\begin{enumerate}
		\item[(a)] $A(I^\ell_i,I\setminus I^\ell_i)$ is called an HSS block row and $A(I\setminus I^\ell_i,I^\ell_i)$ is called an HSS block column for $i = 1,\ldots,2^{\ell}$, $\ell = 1,\ldots,p$.
		\item[(b)] For $k\in \mathbb N$, $A$ is called a $(\mathcal T_p,k)$-HSS matrix if every HSS block row and column of $A$ has rank at most $k$.
		\item[(c)] The HSS rank of $A$ (with respect to $\mathcal T_p$) is the smallest integer $k$ such that $A$ is a $(\mathcal T_p,k)$-HSS matrix.
	
	\end{enumerate}
\end{definition}

\begin{figure}[!ht]
	\centering
	\begin{tikzpicture}[scale=0.3]    
	
	\begin{scope}
	[xshift=0cm]
	\node [above] at (4,8) {$A(I^3_4,I\setminus I^3_4)$};
	\draw (0,0) -- (0,8) -- (8,8) -- (8,0) -- cycle;
	
	\draw[dashed, pattern=north west lines, pattern color=blue] (0,4) rectangle (3,5);
	\draw[dashed, pattern=north west lines, pattern color=blue] (4,4) rectangle (8,5);
	\draw (6,1) -- (8,1);
	\draw (4,2) -- (8,2);
	\draw (4,3) -- (6,3);
	\draw (0,4) -- (8,4);
	\draw (2,5) -- (4,5);
	\draw (0,6) -- (4,6);
	\draw (0,7) -- (2,7);
	\draw (1,6) -- (1,8);
	\draw (2,4) -- (2,8);
	\draw (3,4) -- (3,6);
	\draw (4,0) -- (4,8);
	\draw (5,2) -- (5,4);
	\draw (6,0) -- (6,4);
	\draw (7,0) -- (7,2);
	\end{scope}
	
	\begin{scope}
	[xshift=15cm]
	\node [above] at (4,8) {$A(I\setminus I^2_3,I^2_3)$};
	\draw (0,0) -- (0,8) -- (8,8) -- (8,0) -- cycle;
	\draw[dashed, pattern=north west lines, pattern color=blue] (4,4) rectangle (6,8);
	\draw[dashed, pattern=north west lines, pattern color=blue] (4,0) rectangle (6,2);
	
	\draw (6,1) -- (8,1);
	\draw (4,2) -- (8,2);
	\draw (4,3) -- (6,3);
	\draw (0,4) -- (8,4);
	\draw (2,5) -- (4,5);
	\draw (0,6) -- (4,6);
	\draw (0,7) -- (2,7);
	\draw (1,6) -- (1,8);
	\draw (2,4) -- (2,8);
	\draw (3,4) -- (3,6);
	\draw (4,0) -- (4,8);
	\draw (5,2) -- (5,4);
	\draw (6,0) -- (6,4);
	\draw (7,0) -- (7,2);
	\end{scope}
	
	\end{tikzpicture}
	
	\caption{ Image taken from \cite{kressner2017low}: illustration of an HSS block row and an HSS block column for a cluster tree of depth $3$.}\label{fig:hsscolrow}
\end{figure}
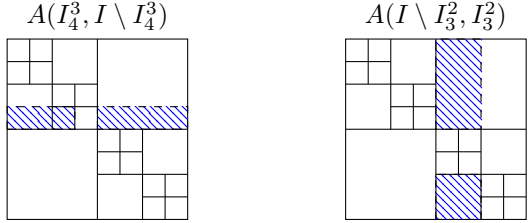

\begin{paragraph}{Matlab class} The \lstinline|hss| class provided by the \hm{} uses the following properties to represent an HSS matrix recursively:
	\begin{itemize}
		\item \lstinline|A11| and \lstinline|A22| are \lstinline|hss| instances representing the diagonal blocks (for a non-leaf node); 
		\item \lstinline|U| and \lstinline|V| contain the basis matrices $U_i^{(p)}$ and $V_i^{(p)}$ for a leaf node and are empty otherwise;
		\item \lstinline|Rl| and \lstinline|Rr| are such that \lstinline|[Rl; Rr]| is the translation operator $R_{U,i}^{(\ell)}$, \\
		\lstinline|Wl| and \lstinline|Wr| are such that \lstinline|[Wl; Wr]| is the translation operator $W_{U,i}^{(\ell)}$ \\
		(note that \lstinline|Rl|, \lstinline|Rr|, \lstinline|Wl|, \lstinline|Wr| are empty for the top node or a leaf node);
		\item \lstinline|B12|, \lstinline|B21| contain the matrices $S_{i,j}^{(\ell)}$, $S_{j,i}^{(\ell)}$ for a non-leaf node;
		\item \lstinline|D| is either a dense matrix representing the whole matrix (for a leaf node) or empty.
	\end{itemize}
	Using the \lstinline|hss| class, $\mathcal O(nk)$ memory is needed to represent 
	a matrix of HSS rank $k$. 
\end{paragraph}

\subsection{Appearance of HODLR and HSS matrices}

Our toolbox is most effective for matrices of small HODLR or HSS rank. In some cases, this property is evident, e.g., for matrices with particular sparsity patterns such as banded matrices. However, there are numerous situations of interest in which the matrix is dense but still admits a highly accurate approximation by a matrix of small HODLR or HSS rank. In particular, this is the case for the discretization of (kernel) functions and integral operators under certain regularity conditions; see~\cite{Borm2010,Hackbusch2015,Hackbusch2004,massei2018fast} for examples.

When manipulating HODLR and HSS matrices, using the functionality of the toolbox, it would be desirable that the off-diagonal low-rank structure is (approximately) preserved.
For more restrictive formats, such as semi- and quasiseparable matrices, the low-rank structure is preserved exactly by certain matrix factorizations and inversion; see the monographs~\cite{Eidelman2014a,Eidelman2014b,Vandebril2008,Vandebril2008a}. While the HSS rank is also preserved by inversion, the same does not hold for the HODLR rank. Often, additional properties are needed in order to show that the HODLR and HSS formats are (approximately) preserved under arithmetic operations; see~\cite{BebendorfH2003,Faustmann2017,Grasedyck2004a,kressner2017low,Massei2017quasi,Gavrilyuk}.

\section{Construction of HODLR / HSS representation}\label{sec:construction}

Even when it is known that a given matrix can be represented or accurately approximated in the HODLR or HSS formats, it is by no means a trival task
to construct such structured representations efficiently.
Often, the construction needs to be tailored to the problem at hand, especially if one aims at handling large-scale matrices and thus needs to bypass the ${\cal O}(n^2)$ memory needed for the explicit dense representation of the matrix.
The \hm{} provides several constructors (summarized in Table~\ref{tab:constructors} below)
trying to capture the most typical situations for which the HODLR and HSS formats are utilized.
The constructors and, more generally, the \hm{} support both real and complex valued matrices.

\subsection{Parameter settings for constructors}

The output of the constructors depend on a number of parameters. In particular, the truncation tolerance $\epsilon$, which guides the error in the spectral norm when approximating a given matrix by a HODLR/HSS matrix, can be set with the following commands:
\begin{lstlisting}[mathescape=true]
hodlroption('threshold', $\epsilon$)
hssoption('threshold', $\epsilon$)
\end{lstlisting}
The default setting is $\epsilon=10^{-12}$ for both formats.

When approximating with a HODLR matrix, the rank of each off-diagonal block $A(I_i^p, I_j^p)$ is chosen such that the spectral norm of the approximation error is bounded by $\epsilon$ times $\|A(I_i^p, I_j^p)\|_2$ or an estimate thereof. For example, when using the truncated singular value decomposition (SVD) for low-rank truncation, this means that $k$ is determined by the number of singular values larger than $\epsilon \|A(I_i^p, I_j^p)\|_2$; see, e.g.,~\cite{Golub2013}. Ensuring such a (local) truncation error guarantees that the overall approximation for the whole matrix $A$ is bounded by $\mathcal O(\epsilon \log (n) \|A\|_2)$ in the spectral norm, see~\cite[Lemma 6.3.2]{Hackbusch2015} and~\cite[Theorem 2.2]{Bini2017}. 

When approximating with an HSS matrix, the tolerance $\epsilon$ guides the approximation error when compressing HSS block rows and columns. The interplay between local and global approximation errors is more subtle and depends on the specific procedure. In general, the global approximation error stays proportional to $\epsilon$. Specific results for the Frobenius and spectral norms can be found in~\cite[Corollary 4.3]{Xi2014} and~\cite[Theorem 4.7]{kressner2017low}, respectively. See also \cite[Theorem~5.30, Corollary~5.31]{Borm2010}
for results on the more general class of $\mathcal H^2$-matrices.

By default, our constructors determine the cluster tree $\mathcal T_p$ by splitting the row and column index sets as equally as possible until 
a minimal block size $n_{\min}$ is reached. More specifically, an index set $\{1,\dots,n\}$ is split into $\{1,\dots,\lceil\frac n2\rceil\}\cup \{\lceil\frac n2\rceil+1,\dots,n\}$.  The default value for $n_{\min}$ is $256$; this value can be adjusted by calling \lstinline|hodlroption('block-size', nmin)| and \lstinline|hssoption('block-size', nmin)|. In Section~\ref{sec:cluster} below, we explain how non-standard cluster trees can be specified.

\subsection{Construction from dense or sparse matrices}\label{sec:dense-sparse}

The HODLR/HSS approximation of a given dense or sparse matrix $A\in\mathbb C^{n\times n}$ is obtained via 
\begin{lstlisting}
hodlrA = hodlr(A);
hssA = hss(A);
\end{lstlisting}

In the following, we discuss the algorithms behind these two commands.

\begin{paragraph}{{\rm \lstinline|hodlr|} for dense $A$}
	To obtain a HODLR approximation, the Householder QR decomposition with column pivoting \cite{Businger} is applied to each off-diagonal block. The algorithm is terminated when an upper bound for the spectral norm of the remainder is below $\epsilon$ times the maximum pivot element. Although there are examples for which such a procedure severely overestimates the (numerical) rank~\cite[Sec. 5.4.3]{Golub2013}, this rarely happens in practice.
	If $k$ denotes the HODLR rank of the output, this procedure has complexity $\mathcal O(kn^2)$. 
	Optionally, the truncated SVD mentioned above instead of QR  with pivoting can be used for compression.  The following commands are used to switch between both methods:
	\begin{lstlisting}
	hodlroption('compression', 'svd');
	hodlroption('compression', 'qr');
	\end{lstlisting}
	
\end{paragraph}

\begin{paragraph}{{\rm \lstinline|hodlr|} for sparse $A$}
	The two sided Lanczos method~\cite{Simon2000}, which only requires matrix-vector multiplications with an off-diagonal block and its (Hermitian) transpose, combined with recompression~\cite{Ballani2016}
	is applied to each off-diagonal block. The method uses the heuristic stopping criterion described in \cite[page 173]{Ballani2016} with threshold $\epsilon$ times an estimate of the spectral norm of the block under consideration. Letting $k$ again denote the HODLR rank of the output and assuming that Lanczos converges in $\mathcal O(k)$ steps, this procedure has complexity $\mathcal O(k^2n\log(n)+k n_z)$, where $n_z$ denotes the number of nonzero entries of $A$.
\end{paragraph}

\begin{paragraph}{{\rm \lstinline|hss|} for dense $A$}
	The algorithm described in \cite[Algorithm 1]{Xia2010} is used, which essentially applies low-rank truncation to every HSS block row and column starting from the leaves to the root of the cluster tree and ensuring the nestedness of the factors~\eqref{eq:HSS_nestedness}. As for \lstinline|hodlr|, one can choose between QR with column pivoting or the SVD (default) for low-rank truncation via \lstinline|hssoption|.
	
	Letting $k$ denote the HSS rank of the output, the complexity of this procedure is $\mathcal O(kn^2)$. 
\end{paragraph}

\begin{paragraph}{{\rm \lstinline|hss|} for sparse $A$}
	The algorithm described in \cite{Martinsson2011} is used, which is based on the randomized SVD~\cite{Halko2011} and involves matrix-vector products with the the entire matrix $A$ and its (Hermitian) transpose.
	We use $10$ random vectors for deciding whether to terminate the randomized SVD, which ensures an accuracy of $\mathcal O(\epsilon)$ with probability at least $1-6\cdot 10^{-10}$~\cite[Section 2.3]{Martinsson2011}. Assuming that $\mathcal O(k)$ random vectors are needed in total, the complexity of this procedure is $\mathcal O(k^2n+kn_z)$.
	%We remark that the convergence in $\mathcal O(k)$ steps is proved to happen with high probability. More precisely,
\end{paragraph}

\subsection{Construction from handle functions}

We provide constructors that access $A$ indirectly via handle functions.

For HODLR, given a handle function \texttt{@(I, J) Aeval(I,J)} that provides the submatrix $A(I,J)$ given row and column indices $I,J$, the command
\begin{lstlisting}
hodlrA = hodlr('handle', Aeval, n, n);
\end{lstlisting}
returns a HODLR approximaton of $A$.
We apply \emph{Adaptive Cross Approximation } (ACA) with partial pivoting \cite[Algorithm 1]{Borm2006}  to approximate each off-diagonal block. The global tolerance $\epsilon$ is used as threshold for the (heuristic) stopping criterion of ACA. 

The HSS constructor uses two additional handle functions \texttt{@(v) Afun(v)} and \texttt{@(v) Afunt(v)} for matrix-vector produces with $A$ and $A^*$, respectively.
The command
\begin{lstlisting}
hssA = hss('handle', Afun, Afunt, Aeval, n, n);
\end{lstlisting}
returns an HSS approximation using the algorithm for sparse matrices discussed in Section~\ref{sec:dense-sparse}.

\subsection{Construction from structured matrices}
When $A$ is endowed with a structure that allows its description with a small number of parameters, it is sometimes possible to efficiently obtain a HODLR/HSS approximation.
All such constructors provided in \hm{} have the syntax 
\begin{lstlisting}
hodlrA = hodlr(structure, ...);
hssA = hss(structure, ...);
\end{lstlisting}
where \lstinline|structure| is a string describing the properties of $A$. The following options are provided:
\begin{description}
	\item[\mytexttt{'banded'}] Given a banded matrix (represented as a sparse
	matrix on input) with lower and upper bandwidth $b_l$ and $b_u$, this constructor returns an exact $(\mathcal T_p, \max\{ b_u,b_l\})$-HODLR or $(\mathcal T_p, b_l+b_u)$-HSS representation of the matrix. For instance,
	\begin{lstlisting}
	h = 1/(n - 1);
	A = spdiags(ones(n,1) * [1 -2 1], -1:1, n, n) / h^2;
	hodlrA = hodlr('banded', A);
	\end{lstlisting}
	returns a representation
	of the 1D discrete Laplacian; see also~\eqref{eq:laplace} below. 
	\item[\mytexttt{'cauchy'}] Given two vectors $x$ and $y$ representing a Cauchy matrix $A$ with entries $a_{ij}=\frac{1}{x_i+y_j}$, the commands
	\begin{lstlisting}
	hodlrA = hodlr('cauchy', x, y);
	hssA = hss('cauchy', x, y);
	\end{lstlisting}
	return a HODLR/HSS approximation of $A$.
	For the HODLR format, the construction relies on the \lstinline|'handle'| constructor described above. The HSS representation is obtained by first performing a HODLR approximation and then converting to the HSS format; see Section~\ref{sec:conversion}.
	\item[\mytexttt{'diagonal'}] Given the diagonal $v$ of a diagonal matrix, the commands
	\begin{lstlisting}
	hodlrA = hodlr('diagonal', v);
	hssA = hss('diagonal', v);
	\end{lstlisting}
	return an exact representation with HODLR/HSS ranks equal to $0$.
	\item[\mytexttt{'eye'}] Given $n$, this constructs a HODLR/HSS representations of the $n\times n$ identity matrix.
	\item[\mytexttt{'low-rank'}] Given $A=UV^*$ in terms of its (low-rank) factors $U,V$ with $k$ columns, this returns an exact 
	$(\mathcal T_p, k)$-HODLR or $(\mathcal T_p, k)$-HSS representation.
	
	\item[\mytexttt{'ones'}]
	Constructs a HODLR/HSS representations of the matrix of all ones. As this is a rank-one matrix, this represents a special case of \lstinline|'low-rank'|. 
	\item[\mytexttt{'toeplitz'}]
	Given the first column $c$ and the first row $r$ of a Toeplitz matrix $A$, the following lines construct HODLR and HSS approximations of $A$:
	\begin{lstlisting}
	hodlrA = hodlr('toeplitz', c, r);
	hssA = hss('toeplitz', c, r);
	\end{lstlisting}	
	For a Toeplitz matrix, the off-diagonal blocks $A_{12},A_{21}$ on the first level in the cluster tree already contain most of the required information. Indeed, all off-diagonal blocks are sub-matrices of these two. To obtain a HODLR approximation we first construct low-rank approximations of $A_{12},A_{21}$ using the two-sided Lanczos algorithm discussed above, combined with FFT based fast matrix-vector multiplication. For all deeper levels, low-rank factors of the off-diagonal blocks are simply obtained by restriction. This constructor is used in Section~\ref{sec:2dfractional} to discretize fractional differential operators.
	For obtaining a HSS approximation, we rely on the \lstinline|'handle'| constructor. 
	\item[\mytexttt{'zeros'}] Constructs a HODLR/HSS representations of the zero matrix.
\end{description}

\begin{table}
	\centering
	\begin{tabular}{|c|c|c|}
		\hline
		Constructor&HODLR complexity&HSS complexity\\
		\hline
		Dense&$\mathcal O(kn^2)$&$\mathcal O(kn^2)$\\
		\hline
		Sparse&$\mathcal O(k^2n\log(n)+kn_z)$&$\mathcal O(k^2n+kn_z)$\\
		\hline
		\lstinline|'banded'|&$\mathcal O(kn\log n)$&$\mathcal O(kn)$\\
		\hline
		\lstinline|'cauchy'|&$\mathcal O(kn\log n)$& $\mathcal O(kn\log n)$\\
		\hline
		\lstinline|'diagonal'|&$\mathcal O(n)$&$\mathcal O(n)$\\
		\hline
		\lstinline|'eye'|&$\mathcal O(n)$&$\mathcal O(n)$\\
		\hline
		\lstinline|'handle'|&$\mathcal O(\mathcal C_1kn\log n)$&$\mathcal O(\mathcal C_1 n + \mathcal C_2k)$\\
		\hline
		\lstinline|'low-rank'|&$\mathcal O(kn\log n)$&$\mathcal O(k n)$\\
		\hline
		\lstinline|'ones'|&$\mathcal O(n\log n)$&$\mathcal O(n)$\\
		\hline
		\lstinline|'toeplitz'|&$\mathcal O(kn\log n)$&$\mathcal O(kn\log n)$\\
		\hline
		\lstinline|'zeros'|&$\mathcal O(n)$&$\mathcal O(n)$\\
		\hline
	\end{tabular}
	\caption{Complexities of constructors; The symbol $\mathcal C_1$ denotes the complexity of computing a single entry of the matrix through the handle function. The symbol $\mathcal C_2$ indicates the cost of the matrix-vector multiplication by $A$ or $A^*$.}\label{tab:constructors}
\end{table}
\subsection{Conversion between formats}\label{sec:conversion}

The \hm{} functions \lstinline|hodlr2hss| and \lstinline|hss2hodlr| convert between the HODLR and HSS formats. An HSS matrix is converted into a HODLR matrix by simply building explicit low-rank factorizations of the off-diagonal blocks from their implicit nested representation in the HSS format. This is done recursively by using the translation operators and the core blocks  $S_{i,j}^{(\ell)}$ with a cost of $\mathcal O(kn\log n)$ operations. A HODLR matrix is converted into an HSS matrix by first incorporating the (dense) diagonal blocks and then performing a sequence of low-rank updates in order to add the off-diagonal blocks that appear on each level. In order to keep the HSS rank as low as possible, recompression is performed after each sum; see also Section~\ref{sec:compression} below. The whole procedure has a cost of $\mathcal O(k^2n\log n)$ where $k$ is the HSS rank of the argument. 

HODLR and HSS matrices are converted into dense matrices using the \lstinline|full| function. In analogy to the sparse format in \matlab, an arithmetic operation between different types of structure always results in the ``less structured'' format. As we consider
HSS to be the more structured format compared to HODLR, this induces the hierarchy reported in Table~\ref{tab:hierarchy}. 
\begin{table}
	\centering
	\begin{tabular}{c|ccc}
		\hline
		$\mathrm{op}$&HSS&HODLR&Dense\\
		\hline
		HSS&HSS&HODLR&Dense\\
		HODLR&HODLR&HODLR&Dense\\
		Dense&Dense&Dense&Dense
	\end{tabular}
	\caption{Format of the outcome of a matrix-matrix operation $\mathrm{op}\in\{+,-,*,\backslash, / \}$ depending on the structure of the two inputs. }\label{tab:hierarchy}
\end{table} 

Some matrices, like inverses of banded matrices, are HODLR and approximately sparse at the same time. In such situations, it can be of interest to 
convert a HODLR matrix into a sparse matrix by neglecting entries below a certain tolerance. The overloaded function \lstinline|sparse| effects this conversion efficiently by only considering those off-diagonal entries for which the corresponding rows of the low-rank factors are sufficiently large. In the following example, entries below $10^{-8}$ are neglected.
\begin{lstlisting}
n = 2^(14);
A = spdiags( ones(n, 1) * [1 3 -1], -1:1, n, n);
hodlrA = hodlr(A); hodlrA = inv(hodlrA);
spA = sparse(hodlrA, 1e-8);
fprintf('Bandwidth: %d, Error = %e\n',...
bandwidth(spA), normest(spA * A - speye(n), 1e-4));
Bandwidth: 14, Error = 3.131282e-08
\end{lstlisting} 
For an HSS matrix, the \lstinline|sparse| function proceeds indirectly via first converting to the HODLR format by means of \lstinline|hss2hodlr|.
%and then obtains a sparse matrix by applying the \lstinline|sparse| function for HODLR.

In summary, the described functionality allows to switch back and forth between HODLR, HSS, and sparse formats. 

\subsection{Auxiliary functionality}

The \hm{} contains several functions that make it convenient to work with HODLR and HSS matrices. For example, the Matlab functions \lstinline|diag|, \lstinline|imag|, \lstinline|real|, \lstinline|trace|, \lstinline|tril|, \lstinline|triu| have been overloaded to compute the corresponding quantities for HODLR/HSS matrices.
We also provide the command \lstinline|spy| to inspect the structure of an \lstinline|hodlr| or \lstinline|hss| instance by plotting the ranks of off-diagonal blocks in the given partitioning. Two examples for the output of \lstinline|spy(A)| are given in Figure~\ref{fig:spy}.

\subsection{Non standard cluster trees}\label{sec:cluster}

A cluster tree $\mathcal T_p$ is determined by the partitioning of the index set on the deepest level and can thus be represented by the vector $c:=[n_{1}^{(p)},\dots,n_{2^p}^{(p)}]$; see Definition~\ref{def:clustertree}. For example, the cluster tree in Figure~\ref{fig:blockpart} is represented by
$c=[1,2,\ldots,8]^T$.

Note that it is possible to construct cluster trees for which the index sets are not equally partitioned on one level. In fact, some index sets can be empty.  For instance, the  cluster tree  in Figure~\ref{fig:cluster2} is represented by the vector $c=[2,4,8,8]^T$.
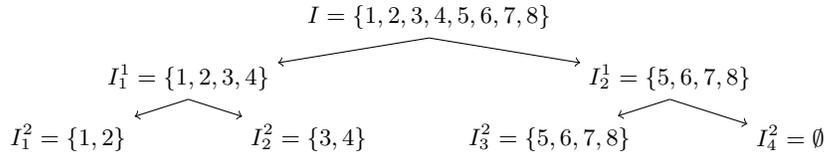
\begin{figure}
	\centering
	\begin{tikzpicture}[scale=0.8] \small
	\coordinate (T18) at (0,0);  
	\coordinate (T14) at (-4,-1);
	\coordinate (T58) at (4,-1);
	\coordinate (T12) at (-6,-2);
	\coordinate (T34) at (-2,-2);
	\coordinate (T56) at (2,-2);
	\coordinate (T78) at (6,-2);
	\coordinate (T1) at (-7,-3);
	\coordinate (T2) at (-5,-3);
	\coordinate (T3) at (-3,-3);
	\coordinate (T4) at (-1,-3);
	\coordinate (T5) at (1,-3);
	\coordinate (T6) at (3,-3);
	\coordinate (T7) at (5,-3);
	\coordinate (T8) at (7,-3);
	
	\node (N18) at (T18) {$I = \{1,2,3,4,5,6,7,8\}$};
	\node (N14) at (T14) {$I_1^1 = \{1,2,3,4\}$};
	\node (N58) at (T58) {$I_2^1 = \{5,6,7,8\}$};
	\node (N12) at (T12) {$I_1^2 = \{1,2\}$};
	\node (N34) at (T34) {$I_2^2 = \{3,4\}$};
	\node (N56) at (T56) {$I_3^2 = \{5,6,7,8\}$};
	\node (N78) at (T78) {$I_4^2 = \emptyset$};
	
	\draw[->] (N18.south) -- (N14);
	\draw[->] (N18.south) -- (N58);  
	\draw[->] (N14.south) -- (N12);
	\draw[->] (N14.south) -- (N34);    
	\draw[->] (N58.south) -- (N56);  
	\draw[->] (N58.south) -- (N78);
	
	\end{tikzpicture}
	\caption{Example of a cluster tree with a leaf node containing an empty index set.}\label{fig:cluster2}
\end{figure}

The vector $c$ is used inside the \hm{} to specify a cluster tree. For all constructors discussed above, an optional argument can be provided to specify the cluster tree for the rows and columns. For those constructors that also allow for rectangular matrices (see below), different cluster trees can be specified for the rows and columns.
For example, the partitioning of Figure~\ref{fig:cluster2} can be imposed on  an $8\times 8$ matrix $A$  as follows:
\begin{lstlisting}
c = [2 4 8 8];
hodlrA = hodlr(A, 'cluster', c);
spy(hodlrA);
\end{lstlisting}
The output of \lstinline|spy| for such a matrix is reported in Figure~\ref{fig:spy}.
\begin{figure}
	\centering
	\includegraphics[width=.47\linewidth]{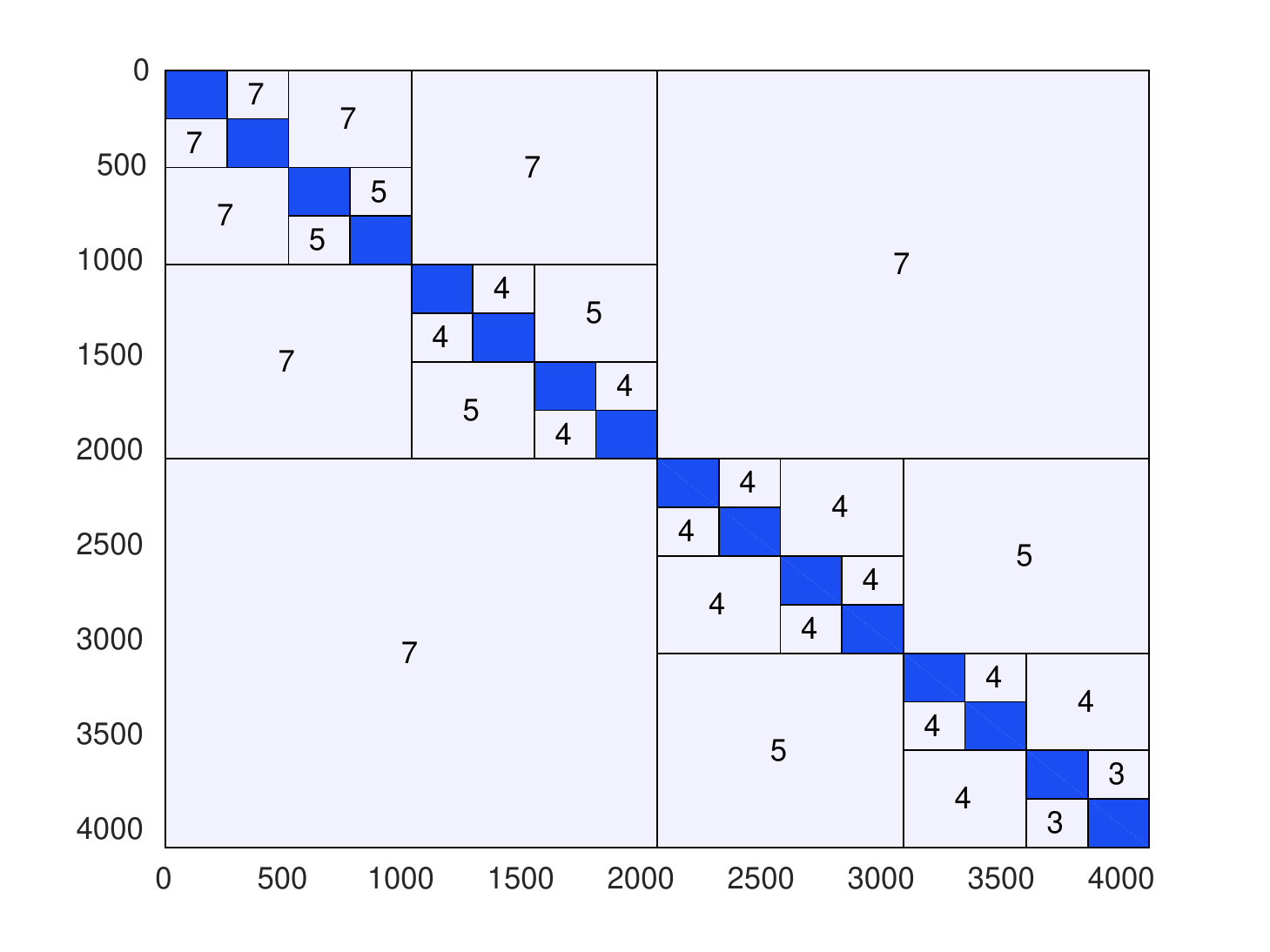} \includegraphics[width=.47\linewidth]{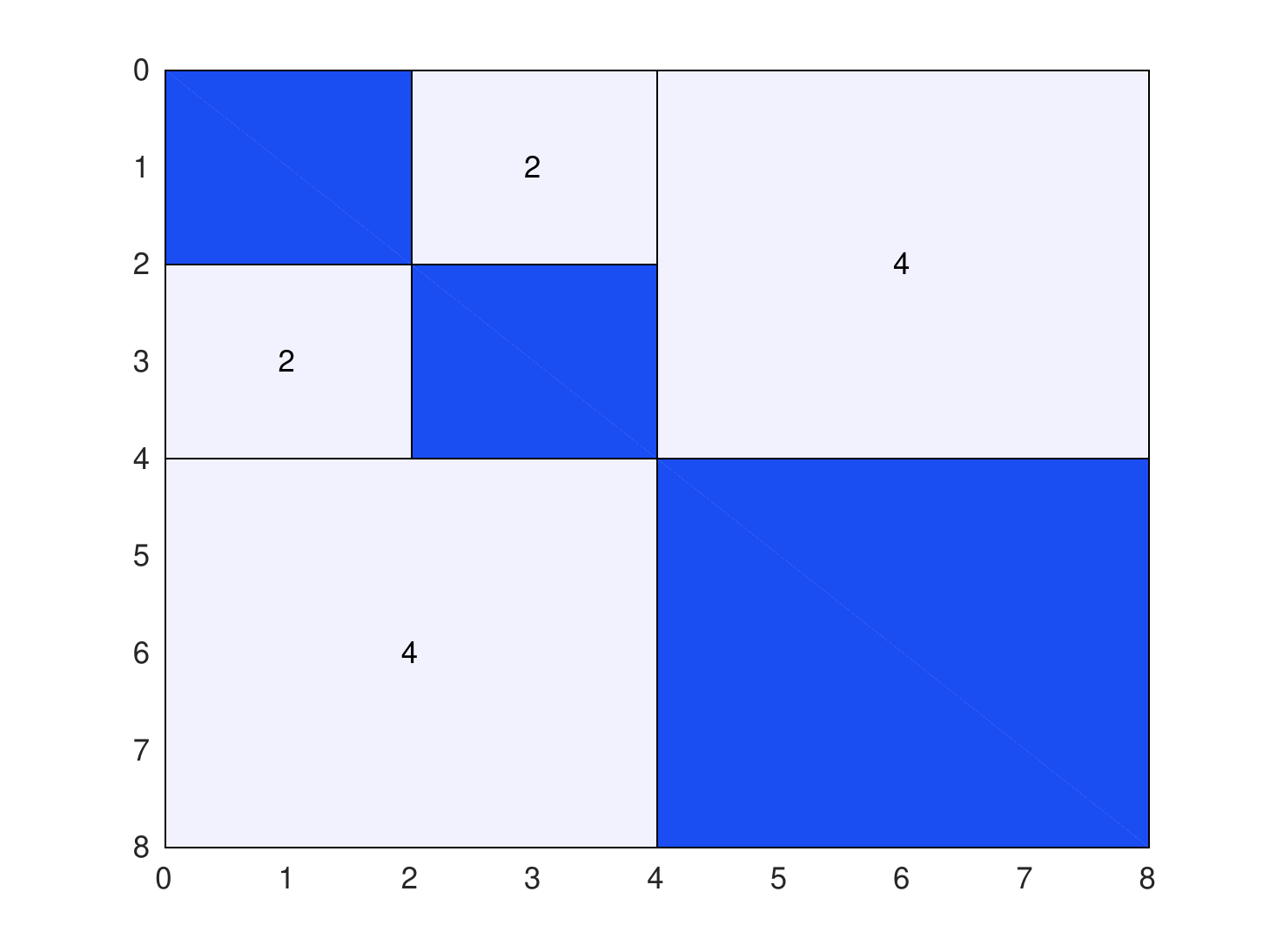}
	\caption{Output of the command \lstinline|spy| for a Cauchy matrix (left) and for a matrix with a non-standard cluster tree (right).}\label{fig:spy}
\end{figure}

The vectors describing the row and column clusters of a given HODLR/HSS matrix can be retrieved using the \lstinline|cluster| command.

\subsection{Rectangular matrices}

The \hm{} also allows to create rectangular HODLR/HSS matrices by means of the dense/sparse constructors or one of the following arguments for the constructor: \lstinline|'cauchy'|, \lstinline|'handle'|, \lstinline|'low-rank'|, \lstinline|'ones'|, \lstinline|'toeplitz'|,\lstinline|'zeros'|. This requires to build two cluster trees, one for the row indices and one for the column indices. If these clusters are not specified, they are built in the default way discussed in Section~\ref{sec:dense-sparse}, such that the children of each node have nearly equal cardinality.
The procedure is carried out simultaneously for the row and column cluster trees and it stops either when both index sets are smaller than the minimal block size or one of the two reaches cardinality $1$. In particular, this ensures that the returned row and column cluster trees have the same depth. 

We remark that some operations for a HODLR/HSS matrix are only available when the row and column cluster trees are equal (which in particular implies that the matrix is square), such as the solution of linear systems, matrix powers, and the determinant.

\section{Arithmetic operations} \label{sec:arithmetic}

The HODLR and HSS formats allow to carry out several arithmetic operations efficiently; a fact that greatly contributes to the versatility of these formats in applications. In this section, we first illustrate the design of fast operations for matrix-vector products and then give an overview over the operations provided in the \hm, 
many of which have already been described in the literature. However, we also provide a few operations that are new to the best of our knowledge.
In particular, this holds for the algorithms for computing $A^{-1}B$ in the HSS format and the Hadamard product in both formats described in Sections~\ref{sec:invatimesb} and~\ref{sec:hadamard}, respectively. As arithmetic operations often increase the HODLR/HSS ranks, it is important to combine them with recompression, a matter discussed in Section~\ref{sec:compression}.

\subsection{Matrix-vector products}

The block partitioning \eqref{eq:hodler1levelxx} suggests the use of a recursive algorithm for computing the matrix-vector product $Av$. Partitioning $v$ in accordance with the columns of $A$, one obtains
\[
Av = \begin{bmatrix}
A_{11}&A_{12}\\
A_{21}&A_{22}
\end{bmatrix}\begin{bmatrix}
v_1\\ v_2
\end{bmatrix}= \begin{bmatrix}
A_{11} v_1 + A_{12}v_2\\
A_{21} v_1 + A_{22}v_2
\end{bmatrix}.
\]
In turn, this reduces $Av$ to smaller matrix-vector products involving low-rank off-diagonal blocks and diagonal blocks. If $A$ is HODLR, the diagonal and off-diagonal blocks are not coupled and $A_{11}v_1, A_{22}v_2$ are simply computed by recursion. The resulting procedure has complexity $\mathcal O(kn\log n)$, see Figure~\ref{fig:mat-vec} (left).

If $A$ is HSS then the off-diagonal blocks $A_{12},A_{21}$ are not directly available, unless the recursion has reached a leaf.  To address this issue, the following four-step procedure is used; see, e.g., \cite[Section 3]{Chandrasekaran2006}. In Step 1, the (column) cluster tree is traversed from bottom to top in order to 
multiply the right-factor matrices $\big( V_j^{(\ell)} \big)^*$ with the corresponding portions of $v$ via the recursive representation~\eqref{eq:HSS_nestedness}. More specifically, letting
$v(I_{i}^{p})$ denote the restriction of $v$ to a leaf $I_{i}^{p}$, we first compute $v_i^{p}:=(V_i^{(p)})^*v(I_{i}^{p})$ on the deepest level
and then retrieve all quantities $v_i^{\ell}:=(V_i^{(\ell)})^*v(I_{i}^{\ell})$ for $\ell=p-1,\dots,1$ by applying the translation operators $R_{V,i}^{(\ell)}$ in a bottom-up fashion. In Step 2, all core blocks $S_{i,j}^{\ell}$ are applied. In Step 3 -- analogous to Step 1 --
the (row) cluster tree is traversed from top to bottom in order to 
multiply the left-factor matrices $U_i^{(\ell)}$ with the corresponding portions of $v$ via the recursive representation~\eqref{eq:HSS_nestedness}. In Step 4, the contributions from the diagonal blocks are added to the vectors obtained at the end of Step 2.
The resulting procedure has complexity $\mathcal O(kn)$, see Figure~\ref{fig:mat-vec} (right).

\begin{figure}\label{fig:mat-vec}
	\begin{minipage}[t]{.44\linewidth}
		\begin{small}
			\begin{algorithmic}[1]
				\Procedure{HODLR\_matvec}{$A,v$}
				\If{$A$ is dense}
				\State \Return $Av$
				\EndIf
				\State $y_1\gets$ \Call{HODLR\_matvec}{$A_{11},v_1$}
				\State $y_2\gets A_{12}v_2$
				\State $y_3\gets A_{21}v_1$
				\State $y_4\gets$ \Call{HODLR\_matvec}{$A_{22},v_2$}
				\State \Return $ 
				\begin{bmatrix}y_1+y_2\\ y_3+y_4\end{bmatrix}$
				\EndProcedure
			\end{algorithmic}
		\end{small}
	\end{minipage}
	\begin{minipage}[t]{.55\linewidth}
		\begin{small}
			\begin{algorithmic}[1]
				\Procedure{HSS\_matvec}{$A,v$}
				\State On level $\ell =p$ compute 
				
				$v_i^{p}\gets(V_i^{(p)})^*v(I_{i}^{p})$,\quad  $i=1,\dots,2^p$
				
				$d_i^{p}\gets A(I_{i}^{p},I_{i}^{p})v(I_{i}^{p})$		
				\For{$\ell=p-1,\dots,1$, $i=1,\dots,2^\ell$}
				\State $v_i^{\ell}\gets(R_{V,i}^{(\ell)})^*\begin{bmatrix}
				v_{2i-1}^{\ell+1}\\
				v_{2i}^{\ell+1}
				\end{bmatrix}$ 
				\EndFor 
				\For{$\ell=1,\dots,p-1$, $i=1,\dots,2^{\ell}$}
				\State $\begin{bmatrix}
				v_{2i-1}^{\ell}\\
				v_{2i}^{\ell}
				\end{bmatrix}\gets\begin{bmatrix}
				0&S_{2i-1,2i}^{(\ell)}\\
				S_{2i,2i-1}^{(\ell)}&0
				\end{bmatrix}\begin{bmatrix}
				v_{2i-1}^{\ell}\\
				v_{2i}^{\ell}
				\end{bmatrix}$ 
				\EndFor 
				\For{$\ell=1,\dots,p-1$, $i=1,\dots,2^{\ell}$}
				\State $\begin{bmatrix}
				v_{2i-1}^{\ell+1}\\
				v_{2i}^{\ell+1}
				\end{bmatrix}\gets \begin{bmatrix}
				v_{2i-1}^{\ell+1}\\
				v_{2i}^{\ell+1}
				\end{bmatrix}+R_{U,i}^{(\ell)}
				v_{i}^{\ell}$
				\EndFor 
				\State On level $\ell =p$ compute 
				
				$y(I_i^p)\gets U_i^{(p)} v_i^{p}+d_i^{p}$,\quad  $i=1,\dots,2^p$
				\State \Return $y$
				\EndProcedure
			\end{algorithmic}
		\end{small}
	\end{minipage}
	\caption{Pseudo-codes of HODLR matrix-vector product (on the left) and HSS matrix-vector product (on the right).}
\end{figure}

\subsection{Overview of fast arithmetic operations in the \hm}

The fast algorithms for performing matrix-matrix operations, matrix factorizations and solving linear systems are based on extensions of the recursive paradigms discussed above for the matrix-vector product. In the HODLR format the original task is split into subproblems that are solved either recursively or relying on low-rank matrix arithmetic; see, e.g., \cite[Chapter 3]{Hackbusch2015} for an overview and \cite{kressner2018fast} for the QR decomposition. In the HSS format, the algorithms have a tree-based structure and a bottom-to-top-to-bottom data flow, see \cite{Xia2010,sheng}. The HSS solver for linear systems is based on an implicit ULV factorization of the coefficient matrix~\cite{Chandrasekaran2006}. A list of the matrix operations available in the toolbox, with the corresponding complexities, is given in Table~\ref{tab:operations}. In the latter, we assume the HODLR/HSS ranks of the matrix arguments to be bounded by $k$. Moreover, for the matrix-matrix multiplication and factorization of HODLR matrices, repeated recompression is needed to limit rank growth of intermediate quantities and we assume that these ranks stay $\mathcal O(k)$. We refer to~\cite{Doelz2019} for an alternative approach for matrix-matrix multiplication based on the randomized SVD.

\begin{table}
	\centering
	\begin{tabular}{|c|c|c|}
		\hline
		Operation&HODLR complexity&HSS complexity\\
		\hline
		\lstinline|A*v|&$\mathcal O(kn\log n)$&$\mathcal O(kn)$\\
		\lstinline|A\v|&$\mathcal O(k^2n\log^2 n)$&$\mathcal O(k^2n)$\\
		\lstinline|A+B|&$\mathcal O(k^2n\log n)$&$\mathcal O(k^2n)$\\
		\lstinline|A*B|&$\mathcal O(k^2n\log^2 n)$&$\mathcal O(k^2n)$\\
		\lstinline|A\B|&$\mathcal O(k^2n\log^2 n)$&$\mathcal O(k^2n)$\\
		\lstinline|inv(A)|&$\mathcal O(k^2n\log^2 n)$&$\mathcal O(k^2n)$\\
		\lstinline|A.*B| \footnote{The complexity of the Hadamard product is dominated by the recompression stage due to the $k^2$ HODLR/HSS rank of $A\circ B$. Without recompression the cost is $\mathcal O(k^2n\log n)$ for HODLR and $\mathcal O(k^2n)$ for HSS.}&$\mathcal O(k^4n\log n)$&$\mathcal O(k^4n)$\\
		\lstinline| lu(A), chol(A)|&$\mathcal O(k^2n\log^2 n)$&---\\
		\lstinline|ulv(A), chol(A)|&---&$\mathcal O(k^2n)$\\
		\lstinline|qr(A)|&$\mathcal O(k^2n\log^2 n)$&---\\
		compression&$\mathcal O(k^2n\log(n))$&$\mathcal O(k^2n)$\\
		\hline
	\end{tabular}
	\caption{Complexity of arithmetic operations in the \hm; $A,B$ are $n\times n$ matrices with HODLR/HSS rank $k$ and $v$ is a vector of length $n$.}\label{tab:operations}
\end{table}

\subsection{$A^{-1}B$ in the HSS format} \label{sec:invatimesb}

Matrix iterations for solving matrix equations or computing matrix functions~\cite{Higham2008} sometimes involve the computation of $A^{-1}B$ for square matrices $A,B$. Being able to perform this operation in HODLR/HSS arithmetic in turn gives the ability to address large-scale structured matrix equations/functions; see \cite{Bini2017cyclic} for an example.

For HODLR matrices $A,B$, the operation $A^{-1}B$ can be implemented in a relative simple manner, by first computing an LU factorization of $A$ and then applying the factors to $B$; see~\cite{Hackbusch2015}. For HSS matrices $A,B$, this operation is more delicate and in the following we describe an algorithm based on the
ideas behind the fast ULV solvers from~\cite{Chandrasekaran2005,Chandrasekaran2006}.

Our algorithm for computing $A^{-1} B$ performs the following four steps:
\begin{enumerate}
	\item The HSS matrix $A$ 
	is sparsified as $\tilde A = Q^* A Z$ by means of orthogonal transformation $Q$
	acting on the row generators at level $p$, and $Z$ triangularizing
	the diagonal blocks. $B$ is updated
	accordingly by left multiplying it with $Q^*$. 
	\item The sparsified matrix is decomposed as a product $\tilde A = A_1 \cdot A_2$, 
	such that $A_1^{-1}$ is easy to apply to $B$; the
	matrix $A_2$ is, up to permutation, of the form $I \oplus \hat A_2$, 
	where $\hat A_2$ is again HSS with the same tree of $A$, but smaller
	blocks. 
	\item The leaf nodes of $\hat A_2$ are merged, yielding an 
	HSS matrix with $p - 1$ levels. The procedure is recursively
	applied for applying $\hat A_2^{-1}$ to the corresponding rows and columns of $A_1^{-1} Q^* B$. 
	\item Finally, $A^{-1} B$ is recovered by applying the orthogonal
	transformation $Z$ from the left to $A_2^{-1} A_1^{-1} Q^* B$. 
\end{enumerate}
We now discuss the four steps in detail. To simplify the description, we assume that all involved ranks  are equal to $k$,

{\bf Step 1.} For each left basis $U_i^{(p)}$ of the HSS matrix $A$, we compute a 
QL factorization $U_i^{(p)} = Q_i \hat U_i^{(p)}$ with a square unitary matrix $Q_i$, such
that 
\[
\hat U_i^{(p)} = Q_i^* U_i^{(p)} = \begin{bmatrix}
0 \\ \tilde U_i^{(p)}
\end{bmatrix}, \qquad 
\tilde U_i^{(p)} \in \mathbb C^{k \times k}.
\]
We define $Q = Q_1 \oplus \dots \oplus Q_{2^p}$ and, in turn,
the matrix $Q^*A$ takes the shape displayed in the left plot of Figure~\ref{fig:ulv-step1}, 
where $\tilde D_i := Q_i^* D_i$, and $D_i$ are the diagonal blocks of $A$. 
Similarly, we 
consider an orthogonal transformation $Z = Z_1 \oplus \dots \oplus Z_{2^p}$ 
such that each $Q_i^* D_i Z_i$ has the form
\[
Q_i^* D_i Z_i = \begin{bmatrix}
\tilde D_{i,11} & 0 \\
\tilde D_{i,21} & \tilde D_{i,22} \\
\end{bmatrix}, \quad 
\tilde D_{i,11} \text{ lower triangular and }
\tilde D_{i,22} \in \mathbb C^{k \times k}. 
\]
Then $\tilde A := Q^* A Z$ has the sparsity pattern 
displayed in the right plot of  Figure~\ref{fig:ulv-step1}.

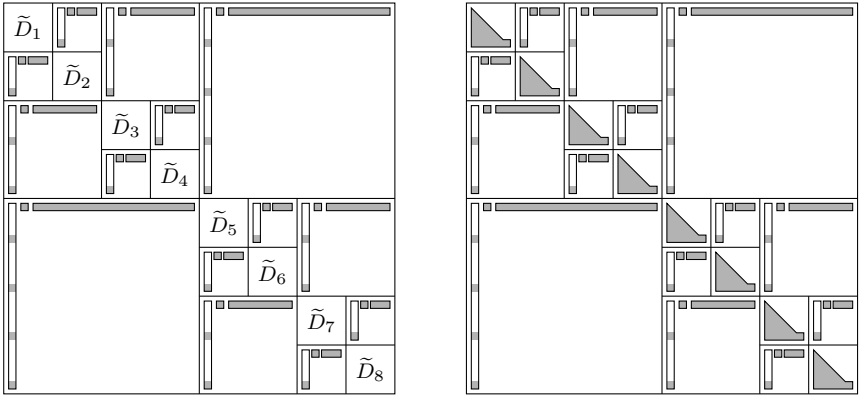
\begin{figure}
	\begin{center}
		\begin{tikzpicture}[scale=0.65]
		\draw (0,0) rectangle (8,8);
		\draw (0,4) -- (8,4); \draw (4,0) -- (4,8);
		\foreach \j in {0, 4} {
			\draw (\j+2, 8-\j) -- (\j+2, 4-\j);
			\draw (8-\j,\j+2) -- (4-\j,\j+2);
		}
		\foreach \j in {0, 2, ..., 6} {
			\draw (\j+1, 8-\j) -- (\j+1, 6-\j);
			\draw (8-\j,\j+1) -- (6-\j,\j+1);
			\pgfmathsetmacro\idx{\j+1};
			\def\idxx{\pgfmathprintnumber{\idx}}
			\node at (\j+.5,7.5-\j) {\footnotesize $\tilde D_{\idxx}$};
			\pgfmathsetmacro\idx{\j+2};
			\def\idxx{\pgfmathprintnumber{\idx}}
			\node at (\j+1.5,6.5-\j) {\footnotesize $\tilde D_{\idxx}$};
		}
		\foreach \s in {0,1} {
			\foreach \i in {0,2,...,6} {
				\fill[DenseBlockColor] (\i+1.1-\s,8-\i-0.9-\s) rectangle (\i+1.25-\s,8-\i-.75-\s);
				\draw (\i+1.1-\s,8-\i-0.9-\s) rectangle (\i+1.25-\s,8-\i-0.1-\s);
				\filldraw[fill=DenseBlockColor] (\i+1.3-\s,8-\i-.25-\s) rectangle (\i+1.45-\s,8-\i-0.1-\s);
				\filldraw[fill=DenseBlockColor] (\i+1.5-\s,8-\i-.25-\s) rectangle (\i+1.9-\s,8-\i-0.1-\s);
			}
		}
		\foreach \s in {0,2} {
			\foreach \i in {0,4} {
				\fill[DenseBlockColor] (\i+2.1-\s,8-\i-1.9-\s) rectangle (\i+2.25-\s,8-\i-1.75-\s);
				\fill[DenseBlockColor] (\i+2.1-\s,8-\i-0.9-\s) rectangle (\i+2.25-\s,8-\i-.75-\s);
				\draw (\i+2.1-\s,8-\i-1.9-\s) rectangle (\i+2.25-\s,8-\i-0.1-\s);
				\filldraw[fill=DenseBlockColor] (\i+2.35-\s, 8-\i-.25-\s) rectangle (\i+2.5-\s,8-\i-0.1-\s);
				\filldraw[fill=DenseBlockColor] (\i+2.6-\s,8-\i-.25-\s) rectangle (\i+3.9-\s,8-\i-0.1-\s);
			}
		}
		\foreach \s in {0,4} {
			\fill[DenseBlockColor] (4.1-\s,8-3.9-\s) rectangle (4.25-\s,8-3.75-\s);
			\fill[DenseBlockColor] (4.1-\s,8-2.9-\s) rectangle (4.25-\s,8-2.75-\s);
			\fill[DenseBlockColor] (4.1-\s,8-1.9-\s) rectangle (4.25-\s,8-1.75-\s);
			\fill[DenseBlockColor] (4.1-\s,8-0.9-\s) rectangle (4.25-\s,8-0.75-\s);
			\draw (4.1-\s,8-3.9-\s) rectangle (4.25-\s,8-0.1-\s);
			\filldraw[fill=DenseBlockColor] (4.35-\s, 8-.25-\s) rectangle (4.5-\s,8-0.1-\s);
			\filldraw[fill=DenseBlockColor] (4.6-\s,8-.25-\s) rectangle (7.9-\s,8-0.1-\s);
		}
		\end{tikzpicture}~~~~~~~~\begin{tikzpicture}[scale=0.65]
		\draw (0,0) rectangle (8,8);
		\draw (0,4) -- (8,4); \draw (4,0) -- (4,8);
		\foreach \j in {0, 4} {
			\draw (\j+2, 8-\j) -- (\j+2, 4-\j);
			\draw (8-\j,\j+2) -- (4-\j,\j+2);
		}
		\foreach \i in {0,1,...,7} {
			\filldraw[fill=DenseBlockColor] 
			(\i+.1,7.9-\i) -- (\i+.1,7.1-\i) -- (\i+.9,7.1-\i) -- (\i+.9,7.25-\i)
			-- (\i+.75,7.25-\i) -- cycle;
		}
		\foreach \j in {0, 2, ..., 6} {
			\draw (\j+1, 8-\j) -- (\j+1, 6-\j);
			\draw (8-\j,\j+1) -- (6-\j,\j+1);
		}
		\foreach \s in {0,1} {
			\foreach \i in {0,2,...,6} {
				\fill[DenseBlockColor] (\i+1.1-\s,8-\i-0.9-\s) rectangle (\i+1.25-\s,8-\i-.75-\s);
				\draw (\i+1.1-\s,8-\i-0.9-\s) rectangle (\i+1.25-\s,8-\i-0.1-\s);
				\filldraw[fill=DenseBlockColor] (\i+1.3-\s,8-\i-.25-\s) rectangle (\i+1.45-\s,8-\i-0.1-\s);
				\filldraw[fill=DenseBlockColor] (\i+1.5-\s,8-\i-.25-\s) rectangle (\i+1.9-\s,8-\i-0.1-\s);
			}
		}
		\foreach \s in {0,2} {
			\foreach \i in {0,4} {
				\fill[DenseBlockColor] (\i+2.1-\s,8-\i-1.9-\s) rectangle (\i+2.25-\s,8-\i-1.75-\s);
				\fill[DenseBlockColor] (\i+2.1-\s,8-\i-0.9-\s) rectangle (\i+2.25-\s,8-\i-.75-\s);
				\draw (\i+2.1-\s,8-\i-1.9-\s) rectangle (\i+2.25-\s,8-\i-0.1-\s);
				\filldraw[fill=DenseBlockColor] (\i+2.35-\s, 8-\i-.25-\s) rectangle (\i+2.5-\s,8-\i-0.1-\s);
				\filldraw[fill=DenseBlockColor] (\i+2.6-\s,8-\i-.25-\s) rectangle (\i+3.9-\s,8-\i-0.1-\s);
			}
		}
		\foreach \s in {0,4} {
			\fill[DenseBlockColor] (4.1-\s,8-3.9-\s) rectangle (4.25-\s,8-3.75-\s);
			\fill[DenseBlockColor] (4.1-\s,8-2.9-\s) rectangle (4.25-\s,8-2.75-\s);
			\fill[DenseBlockColor] (4.1-\s,8-1.9-\s) rectangle (4.25-\s,8-1.75-\s);
			\fill[DenseBlockColor] (4.1-\s,8-0.9-\s) rectangle (4.25-\s,8-0.75-\s);
			\draw (4.1-\s,8-3.9-\s) rectangle (4.25-\s,8-0.1-\s);
			\filldraw[fill=DenseBlockColor] (4.35-\s, 8-.25-\s) rectangle (4.5-\s,8-0.1-\s);
			\filldraw[fill=DenseBlockColor] (4.6-\s,8-.25-\s) rectangle (7.9-\s,8-0.1-\s);
		}
		\end{tikzpicture}
	\end{center}
	\caption{Sparsity patterns of the transformations of $A$ during Step 1.}
	\label{fig:ulv-step1}
\end{figure}

{\bf Step 2.} The matrix $\tilde A$ is decomposed into a product $\tilde A = A_1 \cdot A_2$ as follows. For each block column of $\tilde A$ on the lowest level of recursion we partition $\tilde A(:,I_j^p) =: \big[ C_1, C_2 \big]$ such that $C_2$ has $k$ columns. The corresponding block column of the identity matrix is partitioned analogously: $I(:,I_j^p) =: \big[ E_1, E_2 \big]$. Now, the matrices $A_1,A_2$ are built by setting
\[
A_1(:,I_j^p) := \big[ C_1, E_2 \big], \quad A_2(:,I_j^p) := \big[ E_1, C_2 \big].
\]
The resulting sparsity patterns of these factors are displayed in Figure~\ref{fig:ulv-step2}. 
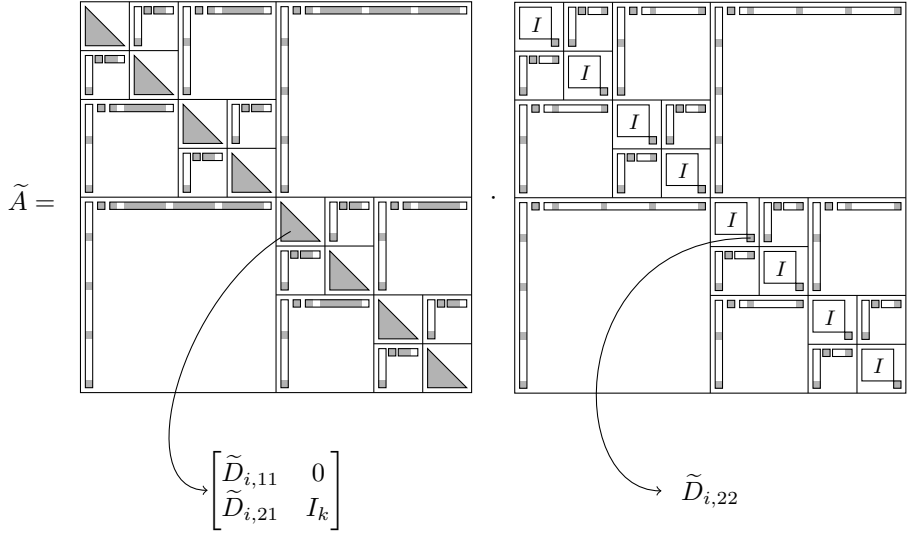
\begin{figure}
	\begin{center}
		\begin{tikzpicture}[scale=0.65]
		\node at (-1,4) {$\tilde A =$};
		\draw (0,0) rectangle (8,8);
		\draw (0,4) -- (8,4); \draw (4,0) -- (4,8);
		\foreach \j in {0, 4} {
			\draw (\j+2, 8-\j) -- (\j+2, 4-\j);
			\draw (8-\j,\j+2) -- (4-\j,\j+2);
		}
		\foreach \i in {0,1,...,7} {
			\filldraw[fill=DenseBlockColor] 
			(\i+.1,7.9-\i) -- (\i+.1,7.1-\i) -- (\i+.9,7.1-\i) -- cycle; 
		}
		\foreach \j in {0, 2, ..., 6} {
			\draw (\j+1, 8-\j) -- (\j+1, 6-\j);
			\draw (8-\j,\j+1) -- (6-\j,\j+1);
		}
		\foreach \s in {0,1} {
			\foreach \i in {0,2,...,6} {
				\fill[DenseBlockColor] (\i+1.1-\s,8-\i-0.9-\s) rectangle (\i+1.25-\s,8-\i-.75-\s);
				\draw (\i+1.1-\s,8-\i-0.9-\s) rectangle (\i+1.25-\s,8-\i-0.1-\s);
				\filldraw[fill=DenseBlockColor] (\i+1.3-\s,8-\i-.25-\s) rectangle (\i+1.45-\s,8-\i-0.1-\s);
				\fill[DenseBlockColor] (\i+1.5-\s,8-\i-.25-\s) rectangle (\i+1.75-\s,8-\i-0.1-\s);
				\draw (\i+1.5-\s,8-\i-.25-\s) rectangle (\i+1.9-\s,8-\i-0.1-\s);
			}
		}
		\foreach \s in {0,2} {
			\foreach \i in {0,4} {
				\fill[DenseBlockColor] (\i+2.1-\s,8-\i-1.9-\s) rectangle (\i+2.25-\s,8-\i-1.75-\s);
				\fill[DenseBlockColor] (\i+2.1-\s,8-\i-0.9-\s) rectangle (\i+2.25-\s,8-\i-.75-\s);
				\draw (\i+2.1-\s,8-\i-1.9-\s) rectangle (\i+2.25-\s,8-\i-0.1-\s);
				\filldraw[fill=DenseBlockColor] (\i+2.35-\s, 8-\i-.25-\s) rectangle (\i+2.5-\s,8-\i-0.1-\s);
				\fill[DenseBlockColor] 
				(\i+2.6-\s,8-\i-.25-\s) rectangle (\i+2.75-\s,8-\i-0.1-\s);
				\fill[DenseBlockColor] 
				(\i+2.9-\s,8-\i-.25-\s) rectangle (\i+3.75-\s,8-\i-0.1-\s);
				\draw (\i+2.6-\s,8-\i-.25-\s) rectangle (\i+3.9-\s,8-\i-0.1-\s);
			}
		}
		\foreach \s in {0,4} {
			\fill[DenseBlockColor] (4.1-\s,8-3.9-\s) rectangle (4.25-\s,8-3.75-\s);
			\fill[DenseBlockColor] (4.1-\s,8-2.9-\s) rectangle (4.25-\s,8-2.75-\s);
			\fill[DenseBlockColor] (4.1-\s,8-1.9-\s) rectangle (4.25-\s,8-1.75-\s);
			\fill[DenseBlockColor] (4.1-\s,8-0.9-\s) rectangle (4.25-\s,8-0.75-\s);
			\draw (4.1-\s,8-3.9-\s) rectangle (4.25-\s,8-0.1-\s);
			\filldraw[fill=DenseBlockColor] (4.35-\s, 8-.25-\s) rectangle (4.5-\s,8-0.1-\s);
			\fill[DenseBlockColor] (4.6-\s,8-.25-\s) rectangle (4.75-\s,8-0.1-\s);
			\fill[DenseBlockColor] (4.9-\s,8-.25-\s) rectangle (5.75-\s,8-0.1-\s);
			\fill[DenseBlockColor] (5.9-\s,8-.25-\s) rectangle (6.75-\s,8-0.1-\s);
			\fill[DenseBlockColor] (6.9-\s,8-.25-\s) rectangle (7.75-\s,8-0.1-\s);
			\draw (4.6-\s,8-.25-\s) rectangle (7.9-\s,8-0.1-\s);
		}
		% Zoom
		\node at (4, -2) {$\begin{bmatrix}
			\tilde D_{i,11} & 0 \\ \tilde D_{i,21}  & I_k
			\end{bmatrix}$};
		\draw[->] (4.3,3.3) .. controls (2,2) and (1,-2) .. (2.6,-2);
		\end{tikzpicture}~\begin{tikzpicture}[scale=0.65]
		\node at (0,4) {$\cdot$};
		\node at (0,-2.8) {}; \node at (0,8) {};
		\end{tikzpicture}\begin{tikzpicture}[scale=0.65]
		\draw (0,0) rectangle (8,8);
		\draw (0,4) -- (8,4); \draw (4,0) -- (4,8);
		\foreach \j in {0, 4} {
			\draw (\j+2, 8-\j) -- (\j+2, 4-\j);
			\draw (8-\j,\j+2) -- (4-\j,\j+2);
		}
		\foreach \i in {0,1,...,7} {
			\filldraw[fill=DenseBlockColor] 
			(\i+.75,7.25-\i) rectangle (\i+.9,7.1-\i);
			\node at (\i+0.45,7.55-\i) {\footnotesize $I$};
			\draw (\i+0.1,7.25-\i) rectangle (\i+0.75,7.9-\i);
		}
		\foreach \j in {0, 2, ..., 6} {
			\draw (\j+1, 8-\j) -- (\j+1, 6-\j);
			\draw (8-\j,\j+1) -- (6-\j,\j+1);
		}
		\foreach \s in {0,1} {
			\foreach \i in {0,2,...,6} {
				\fill[DenseBlockColor] (\i+1.1-\s,8-\i-0.9-\s) rectangle (\i+1.25-\s,8-\i-.75-\s);
				\draw (\i+1.1-\s,8-\i-0.9-\s) rectangle (\i+1.25-\s,8-\i-0.1-\s);
				\fill[DenseBlockColor] (\i+1.5-\s,8-\i-.25-\s) rectangle (\i+1.9-\s,8-\i-0.1-\s);
				\filldraw[fill=DenseBlockColor] (\i+1.3-\s,8-\i-.25-\s) rectangle (\i+1.45-\s,8-\i-0.1-\s);
				\fill[white] (\i+1.5-\s,8-\i-.25-\s) rectangle (\i+1.75-\s,8-\i-0.1-\s);
				\draw (\i+1.5-\s,8-\i-.25-\s) rectangle (\i+1.9-\s,8-\i-0.1-\s);
			}
		}
		\foreach \s in {0,2} {
			\foreach \i in {0,4} {
				\fill[DenseBlockColor] (\i+2.1-\s,8-\i-1.9-\s) rectangle (\i+2.25-\s,8-\i-1.75-\s);
				\fill[DenseBlockColor] (\i+2.1-\s,8-\i-0.9-\s) rectangle (\i+2.25-\s,8-\i-.75-\s);
				\draw (\i+2.1-\s,8-\i-1.9-\s) rectangle (\i+2.25-\s,8-\i-0.1-\s);
				\filldraw[fill=DenseBlockColor] (\i+2.35-\s, 8-\i-.25-\s) rectangle (\i+2.5-\s,8-\i-0.1-\s);
				\fill[DenseBlockColor] (\i+2.6-\s,8-\i-.25-\s) rectangle (\i+3.9-\s,8-\i-0.1-\s);
				\fill[white] 
				(\i+2.6-\s,8-\i-.25-\s) rectangle (\i+2.75-\s,8-\i-0.1-\s);
				\fill[white] 
				(\i+2.9-\s,8-\i-.25-\s) rectangle (\i+3.75-\s,8-\i-0.1-\s);
				\draw (\i+2.6-\s,8-\i-.25-\s) rectangle (\i+3.9-\s,8-\i-0.1-\s);
			}
		}
		\foreach \s in {0,4} {
			\fill[DenseBlockColor] (4.1-\s,8-3.9-\s) rectangle (4.25-\s,8-3.75-\s);
			\fill[DenseBlockColor] (4.1-\s,8-2.9-\s) rectangle (4.25-\s,8-2.75-\s);
			\fill[DenseBlockColor] (4.1-\s,8-1.9-\s) rectangle (4.25-\s,8-1.75-\s);
			\fill[DenseBlockColor] (4.1-\s,8-0.9-\s) rectangle (4.25-\s,8-0.75-\s);
			\draw (4.1-\s,8-3.9-\s) rectangle (4.25-\s,8-0.1-\s);
			\filldraw[fill=DenseBlockColor] (4.35-\s, 8-.25-\s) rectangle (4.5-\s,8-0.1-\s);
			\fill[DenseBlockColor] (4.6-\s,8-.25-\s) rectangle (7.9-\s,8-0.1-\s);
			\fill[white] (4.6-\s,8-.25-\s) rectangle (4.75-\s,8-0.1-\s);
			\fill[white] (4.9-\s,8-.25-\s) rectangle (5.75-\s,8-0.1-\s);
			\fill[white] (5.9-\s,8-.25-\s) rectangle (6.75-\s,8-0.1-\s);
			\fill[white] (6.9-\s,8-.25-\s) rectangle (7.75-\s,8-0.1-\s);
			\draw (4.6-\s,8-.25-\s) rectangle (7.9-\s,8-0.1-\s);
		}
		\node at (4,-2) {$\tilde D_{i,22}$};
		\draw[->] (4.825,3.175) .. controls (1,3.175) and (1,-2) .. (3,-2);
		\node at (0,-2.8) {};
		\end{tikzpicture}
	\end{center}
	\caption{Sparsity patterns of the factors $A_1,A_2$ constructed in Step 2.}
	\label{fig:ulv-step2}
\end{figure}

Letting $\begin{bmatrix}
\tilde D_{i,11} & 0 \\ \tilde D_{i,21}  & I_k
\end{bmatrix}$ denote a diagonal block of $A_1$, we construct the block diagonal matrix 
$A_{1,D}$ with the diagonal blocks $\tilde D_{i,11} \oplus I_k$ for $i = 1,\ldots,2^p$.	
We decompose $A_1 = A_{1,D} + U_A V_A^T$, where $U_A V_A^T$ is a
low-rank factorization of the off-diagonal part of $A_1$. 
In particular, the factors $U_A,V_A$ have $2^p k$ columns and, thanks to their sparsity pattern  (see
the left matrix in Figure~\ref{fig:ulv-step2}), they satisfy the relations
$V_A^T U_A = 0$
and $A_{1,D} U_A = A_{1,D}^{-1} U_A=U_A$. In turn, by the Woodbury matrix identity, we obtain
\[
A_1^{-1} = (A_{1,D} + U_AV_A^T)^{-1} = 
(I - U_A V_A^T) A_{1,D}^{-1}. 
\]
Therefore, computing $A_1^{-1} Q^* B$ comes down to applying the block diagonal matrix $A_{1,D}^{-1}$, followed by
a correction which involves the multiplication with the matrix $U_AV_A^T$ which is $(\mathcal T_p,k)$-HSS. 

{\bf Step 3.}
To apply $A_2^{-1}$ to $A_1^{-1} Q^* B$, we follow the strategy of the fast implicit ULV solver
for linear systems presented in \cite[Section~4.2.3]{Chandrasekaran2005}. After a suitable permutation, 
$A_2$ has the form $I \oplus \hat A_2$, where $\hat A_2$ is a $2^pk\times 2^p k$ HSS matrix (of level $p$) assembled 
by selecting the indices corresponding to the trailing $k\times k$ minors
of the diagonal blocks. As a principal submatrix, the HSS structure of $\hat A_2$ is directly inherited from the one of $\tilde A$ at no cost. Then we call the whole procedure recursively to apply $\hat A_2^{-1}$ to the corresponding rows in 
$A_1^{-1} Q^* B$, which are viewed as a (rectangular) HSS matrix of depth $p-1$.

{\bf Step 4.} To conclude, we apply the block diagonal orthogonal transformation
$Z$, arising from Step 1, to $A_2^{-1} A_1^{-1} Q^* B$. 

\subsection{Hadamard product in the HODLR and in the HSS format} \label{sec:hadamard} 
To carry out the Hadamard (or elementwise) product $A\circ B$ of two HODLR/HSS matrices $A,B$ with the same cluster trees, it is useful to recall the Hadamard product of two low-rank matrices. More specifically, given $U_1B_1V_1^*$ and $U_2B_2V_2^*$ we have that (see Lemma 3.1 in \cite{kressner2017recompression})
\begin{equation}\label{eq:hadamard}
U_1B_1V_1^* \circ U_2B_2V_2^* =(U_1\odot^T U_2)(B_1\otimes B_2)(V_1\odot^T V_2)^*,
\end{equation}
where $\otimes$ denotes the Kronecker product and $\odot^T$ is the \emph{transpose Khatri-Rao product} defined as 
\[
C \in\mathbb C^{n\times q},\quad D\in\mathbb C^{n\times m},\qquad 
C \odot^T D:= \begin{bmatrix}
c_1^T\otimes d_1^T\\
c_2^T\otimes d_2^T\\
\vdots\\
c_n^T\otimes d_n^T\\
\end{bmatrix}\in\mathbb C^{n\times qm},
\]
with $c_i^T$ and $d_i^T$ denoting the $i$th rows of $C$ and $D$, respectively.

Equation \eqref{eq:hadamard} applied to the off-diagonal blocks immediately provides a HODLR representation, where the HODLR ranks multiply, see Figure~\ref{fig:hadamard} (left).

For the HSS format we need to specify how to update the translation operators. To this end we remark that 
\begin{align*}
\left(\begin{bmatrix} \widetilde U_1 & 0 \\ 0 & \widehat U_1 \end{bmatrix} R_{U,1}\right)\odot^T \left(\begin{bmatrix} \widetilde U_2 & 0 \\ 0 & \widehat U_2 \end{bmatrix} R_{U,2}\right)&= \begin{bmatrix} \widetilde U_1 & 0 \\ 0 & \widehat U_1 \end{bmatrix}\odot^T\begin{bmatrix} \widetilde U_2 & 0 \\ 0 & \widehat U_2 \end{bmatrix} ( R_{U,1} \otimes R_{U,2})\\
&= \begin{bmatrix} \widetilde U_1\odot^T \widetilde U_2  & 0 \\ 0 & \widehat U_1 \odot^T \widehat U_2\end{bmatrix}( R_{U,1} \otimes R_{U,2}),
\end{align*}
where we used \cite[Property (4) in Section 2.1]{kressner2017recompression} to obtain the first identity. Putting all the pieces together yields the procedure for the Hadamard product of two HSS matrices, see Figure~\ref{fig:hadamard} (right).

\begin{figure}\label{fig:hadamard}
	\begin{minipage}[t]{.48\linewidth}
		\begin{small}
			\begin{algorithmic}[1]
				\Procedure{HODLR\_hadam}{$A,B$}
				\If{$A,B$ are dense}
				\State \Return $A\circ B$
				\EndIf
				\State $C_{11}\gets$ \Call{HODLR\_hadam}{$A_{11},B_{11}$}
				\State $C_{22}\gets$ \Call{HODLR\_hadam}{$A_{22},B_{22}$}
				\State $C.U_{12}\gets A.U_{12}\odot^T B.U_{12}$
				\State $C.V_{12}\gets A.V_{12}\odot^T B.V_{12}$
				\State $C.U_{21}\gets A.U_{21}\odot^T B.U_{21}$
				\State $C.V_{21}\gets A.V_{21}\odot^T B.V_{21}$
				\State  
				$C:=\begin{bmatrix}C_{11} &C.U_{12}\ C.V_{12}^*\\ C.U_{21}\ C.V_{21}^*&C_{22}\end{bmatrix}$
				\State \Return $C$
				\EndProcedure
			\end{algorithmic}
		\end{small}
	\end{minipage}
	\begin{minipage}[t]{.51\linewidth}
		\begin{small}
			\begin{algorithmic}[1]
				\Procedure{HSS\_hadam}{$A,B$}
				
				\State On level $\ell =p$, for $i=1,\dots, 2^p$ 
				
				$C(I_{i}^{p},I_{i}^{p})\gets A(I_{i}^{p},I_{i}^{p})\circ B(I_{i}^{p},I_{i}^{p})$	
				
				$C.U_i^{(p)}= A.U_i^{(p)}\odot^T B.U_i^{(p)}$	
				
				$C.V_i^{(p)}= A.V_i^{(p)}\odot^T B.V_i^{(p)}$
				\For{$\ell=p-1,\dots,1$}
				\State $C.R_{U,i}^{(\ell)}\gets A.R_{U,i}^{(\ell)}\otimes B.R_{U,i}^{(\ell)}$ 
				\State $C.R_{V,i}^{(\ell)}\gets A.R_{V,i}^{(\ell)}\otimes B.R_{V,i}^{(\ell)}$ 
				\State $C.S_{i,j}^{(\ell)}\gets A.S_{i,j}^{(\ell)}\otimes B.S_{i,j}^{(\ell)}$ 
				\EndFor 
				\State\Return $C$
				\EndProcedure
			\end{algorithmic}
		\end{small}
	\end{minipage}
	\caption{Pseudo-codes of Hadamard product $C = A\circ B$ in the HODLR format (on the left) and the HSS format (on the right). We used the dot notation (e.g., $C.U_{12}$), to distinguish the parameters in the representation of the matrices $A,B,C$. }
\end{figure}
\subsection{Recompression}\label{sec:compression}

The term recompression refers to the following task: 

\noindent Given a $(\mathcal T_p,k)$-HODLR/HSS matrix $A$ and a tolerance $\tau$ we aim at constructing a $(\mathcal T_p,\widetilde k)$-HODLR/HSS matrix $\widetilde A$, with $\widetilde k\leq k$ as small as possible, such that $\norm{A-\widetilde A}_2\leq c\cdot \tau$ for some constant $c$   depending on the format and the cluster tree $\mathcal T_p$.

The recompression of a HODLR matrix applies a well known QR-based procedure~\cite[Section 2.5]{Hackbusch2015} to efficiently recompress each (factorized) off-diagonal block. This procedure ensures that the error in each block is bounded by $\tau$, yielding an overall accuracy $\norm{A-\tilde A}_2\leq p\cdot \tau$.

The recompression of a HSS matrix uses the algorithm from~\cite[Section 5]{Xia2010}, which proceeds in two phases. In the first phase, the HSS representation is transformed to the so-called \emph{proper form} such that all factors $U_i^{(\ell)}$ and $V_i^{(\ell)}$ have orthonormal columns on every level $\ell=1,\dots,p$.  This moves all (near) linear dependencies to the core factors. In the second phase, these core factors are compressed by truncated SVD in a top-to-bottom fashion, while ensuring the nestedness and the proper form of the representation. The output $\widetilde A$ satisfies $\norm{A-\tilde A}_2\leq 2\frac{\sqrt 2^p -1}{\sqrt 2 -1}\cdot \tau\approx \sqrt n/ n_{\mathrm{min}} \tau$; see Appendix~\ref{app:A} for a more detailed description of the algorithm and an error analysis. 

The command \lstinline|compress| carries out the recompression discussed above. Additionally to this explicit involvement, 
most of the algorithms in the toolbox involve recompression techniques implicitly. Performing arithmetic operations often lead to HODLR/HSS representations with ranks larger than necessary to attain the desired accuracy. For instance, if $A$ and $B$ are $(\mathcal T_p,k_A)$-HSS  and $(\mathcal T_p,k_B)$-HSS matrices, respectively, then both $A+B$ and $A\cdot B$ are exactly represented as $(\mathcal T_p,k_A+k_B)$-HSS matrices. However, $k_A+k_B$ is usually an overestimate of the required HSS rank and recompression can be used to limit this rank growth.

When applying recompression to the output $A$ of an arithmetic operation, the toolbox proceeds by first estimating $\norm{A}_2$ by means of the power method on $AA^*$. Then recompression is applied with the tolerance $\tau = \norm{A}_2\cdot \epsilon$, where $\epsilon$ is the global tolerance discussed in Section~\ref{sec:construction}.

The matrix-matrix multiplication in the HODLR format requires some additional care due to the accumulation of low-rank updates from recursive calls \cite{Doelz2019}. Currently, our implementation performs intermediate recompression after each low-rank update with accuracy $\tau$.

\section{Examples and applications} \label{sec:examples}

In this section, we illustrate the use of the \hm{} for a range of applications. 

All experiments have been performed on a server with a 
Xeon CPU E5-2650 v4 running at 2.20GHz; for each test the
running process has been allocated
$8$ cores and 128 GB of RAM. The algorithms are implemented in 
\matlab{} and tested under MATLAB2017a, with MKL BLAS version 11.3.1, 
using the $8$ cores available. 

If not stated otherwise, the parameters $\epsilon$ and $n_{\mathrm{min}}$ are set to their default values.

\subsection{Fast Toeplitz solver}\label{sec:toeplitz}
HSS matrices can be used to design a superfast solver for Toeplitz linear systems. We briefly review the approach in \cite{Xia2012} and describe its implementation that is contained in the function \texttt{toeplitz\_solve} of the toolbox.  

Let $T$ be  an $n\times n$ Toeplitz matrix 
\[
T=\begin{bmatrix}
t_0&t_1&\dots&t_{n-1}\\
t_{-1}&t_0&\ddots&\vdots\\
\vdots&\ddots&\ddots&t_1\\
t_{1-n}&\dots&t_{-1}&t_0
\end{bmatrix}.
\]
In particular, the entries on every diagonal of $T$ are constant and the matrix is completely described by the $2n-1$ real or complex scalars
$t_{1-n},\ldots,t_{n-1}$. 

It is well known that a Toeplitz matrix $T$ satisfies the so called displacement  equation
\begin{equation}\label{eq:displ}
Z_1 T - T Z_{-1}= GH^T
\end{equation} 
where
\[
G = \begin{bmatrix}
1&2t_0\\
0&t_{n-1}+t_{-1}\\
0&t_{n-2}+t_{-2}\\
\vdots&\vdots\\
0&t_1+t_{1-n}
\end{bmatrix},\qquad H = \begin{bmatrix}
t_{1-n}-t_1&0\\
t_{2-n}-t_2&0\\
\vdots&0\\
t_1-t_{n-1}&\vdots\\
0&1
\end{bmatrix},\qquad Z_t:=\begin{bmatrix}
0&t\\
I_{n-1}&0
\end{bmatrix}.
\]
Here, $Z_1$ is a circulant matrix, which is diagonalized by the normalized inverse discrete Fourier transform
\[
\Omega_n=\frac 1{\sqrt n}(\omega_n^{(i-1)(j-1)})_{1\le i,j\le n}, \qquad \Omega_n Z_1\Omega_n^*=\diag(1,\omega_n,\dots,\omega_n^{(n-1)})=:D_1,
\]
with $\omega_n=e^{\frac{2\pi\mathbf i}{n}}$.
Let us call $D_0:=\diag(1,\omega_{2n},\dots,\omega_{2n}^{(n-1)})$. Then, applying $\Omega_n$ from the left and $D_0^*\Omega_n^*$ from the right of \eqref{eq:displ} leads to another displacement equation \cite{Heinig1995}
\begin{equation}\label{eq:sylv}
D_1\mathcal C-\mathcal C D_{-1}=\widehat G\widehat F^T,
\end{equation}
where
\[
\mathcal C = \Omega_nTD_0^*\Omega_n^*,\qquad \widehat G =\Omega_n G,\qquad \widehat F =\Omega_n D_0 H 
\]
and $D_{-1}=\omega_{2n} D_1$. Since the linear coefficients of \eqref{eq:sylv} are diagonal matrices, the matrix $\mathcal C$ is a Cauchy-like matrix of the following form
\begin{equation}\label{eq:cauchy}
\mathcal C=\left(\frac{\widehat G_i \widehat H_j^T}{\omega_{2n}^{2(i-1)}-\omega_{2n}^{2j-1}} \right)_{1\le i,j\le n},
\end{equation} 
where $\widehat G_i,\widehat H_j$ indicate the $i$-th and $j$-th rows of $G$ and $H$ respectively.

The fundamental idea of the superfast solver from~\cite{Xia2012} consists of representing the Cauchy matrix $\mathcal C$ in the HSS format. A linear system $Tx=b$ can be turned into  $\mathcal Cy=z$, with $y=\Omega D_0 x$ and $z=\Omega_n b$. Exploiting the HSS structure of $\mathcal C$ provides an efficient solution of $\mathcal Cy=z$. The solution $x$ of the original system is retrieved with an inverse FFT and a diagonal scaling, which can be
performed with $\mathcal O(n \log n)$ flops. 

The compression of $\mathcal C$ in the HSS format is performed using the \lstinline|'handle'| constructor described in Section \ref{sec:construction}. Indeed, given a vector $x\in\mathbb C^{n}$ we see that $\mathcal C x =\Omega_nTD_0^*\Omega_n^* x$. Therefore, we can evaluate the matrix vector product by means of FFTs and a diagonal  scaling. We assume to have at our disposal an FFT based matrix-vector multiplication for Toeplitz matrices. The latter is used to implement an efficient routine \texttt{C\_matvec} that performs the matrix vector product with $\mathcal C$. Analogously, a routine \texttt{C\_matvec\_transp} for $\mathcal C^*$ is constructed.

The \matlab{} code of \texttt{toeplitz\_solve} (which is included 
in the toolbox) is sketched in the following:
\begin{lstlisting}
function x = toeplitz_solve(c, r, b)
%n, Gh and Fh defined as above
d0  = exp(1i * pi / n .* (0 : n - 1));
d1  = d0 .^ 2;
dm1 = exp(1i * pi / n) * d1;
C = hss('handle',...
@(v)   C_matvec(c, r, d0, v), ...
@(v)   C_matvec_transp(c, r, d0, v), ...
@(i,j) (Gh(i,:) * Fh(j,:)') ./ (d1(i).' - dm1(j)), n, n);
z = ifft(b);
y = C \ z;
x = d0' .* fft(y);
end
\end{lstlisting}
The whole procedure can be carried out in $\mathcal O(k^2 n + k n \log n)$ 
flops, where $k$ is the HSS rank of the Cauchy-like matrix $\mathcal C$. 
Since $k$ is $\mathcal O(\log n)$ \cite{Xia2012},
the solver has a complexity of $\mathcal O(n \log^2 n)$ (assuming that the HSS constructor for the Cauchy-like matrix needs
$\mathcal O(k)$ matrix-vector products). 

We have tested our implementation on the matrices --- named from A to F ---  considered in \cite{Xia2012}. The right hand side is obtained by calling \texttt{randn(n,1)} and we set $\epsilon$ to $10^{-10}$. The timings are reported in Figure~\ref{fig:time1} and the relative residuals $\frac{\norm{Tx-b}_2}{\norm{T}_2\norm{x}_2+\norm{v}_2}$ in Table~\ref{tab:res1}.
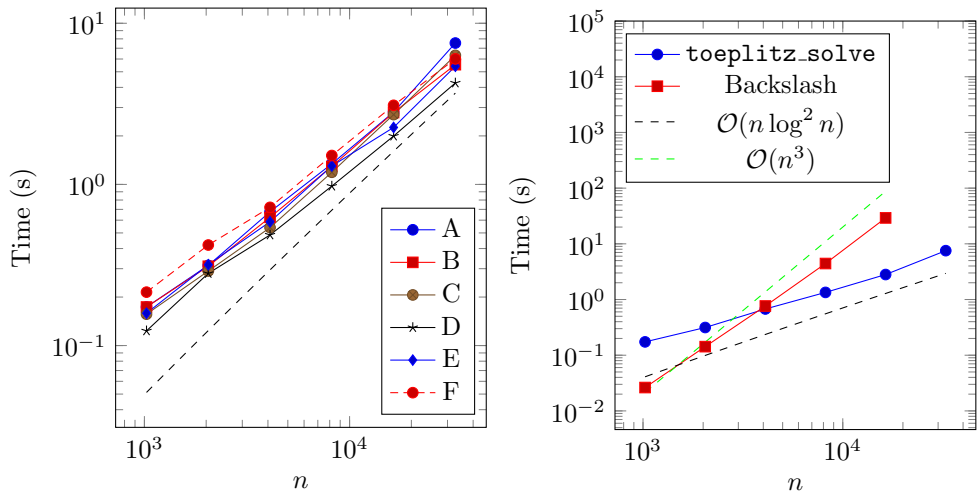
\begin{figure}
	\begin{minipage}{.5\linewidth}
		\begin{tikzpicture}
		\begin{loglogaxis}[xlabel = $n$, ylabel = Time (s),width=\linewidth,
		height=1.1\linewidth,legend pos = south east]
		\addplot table[x index = 0, y index = 1]
		{toeplitz_times.dat};
		\addplot table[x index = 0, y index = 2]
		{toeplitz_times.dat};
		\addplot table[x index = 0, y index = 3]
		{toeplitz_times.dat};
		\addplot table[x index = 0, y index = 4]
		{toeplitz_times.dat};	\addplot table[x index = 0, y index = 5]
		{toeplitz_times.dat};
		\addplot table[x index = 0, y index = 6]
		{toeplitz_times.dat};		
		\addplot[domain=1024:32768, dashed]{5e-7*x*log2(x) * log2(x)};	 		
		\legend{A,B,C,D,E,F}
		\end{loglogaxis}
		\end{tikzpicture}
	\end{minipage}~\begin{minipage}{.49\linewidth}
		\begin{tikzpicture}
		\begin{loglogaxis}[xlabel = $n$, ylabel = Time (s), 
		width=\linewidth,
		height=1.1\linewidth,legend pos = north west, ymax = 1e5]
		\addplot table[x index = 0, y index = 1]
		{toeplitz_times.dat};
		\addplot table[x index = 0, y index = 7]
		{toeplitz_times.dat};
		\addplot[domain=1024 : 32768, dashed] {4e-7 * x * log2(x) * log2(x)};
		\addplot[domain=1024 : 16384, dashed, green] {2e-11 * x^3};       
		\legend{\texttt{toeplitz\_solve},{Backslash},$\mathcal O(n \log^2 n)$, $\mathcal O(n^3)$}
		\end{loglogaxis}
		\end{tikzpicture}
	\end{minipage}
	\caption{Left: Execution time (in seconds) for \texttt{toeplitz\_solve} applied to the Toeplitz matrices A to F from~\cite{Xia2012} and a dashed line indicating a $\mathcal O(n \log^2 n)$ growth. Right: Execution times for \texttt{toeplitz\_solve} vs. Matlab's ``backslash'' applied to the Toeplitz matrix A.}
	\label{fig:time1}
\end{figure}
\begin{table}
	\resizebox{\textwidth}{!}{  
		\begin{filecontents*}{toeplitz_residues.dat}
1024	6.2368e-11	1.49e-09	1.0681e-14	5.0161e-15	3.5227e-15	1.1938e-10
2048	1.1356e-10	1.2227e-09	1.3077e-12	2.9503e-15	9.4668e-15	1.6979e-10
4096	9.0437e-11	1.584e-09	5.2261e-13	9.4355e-16	3.9095e-14	1.3e-10
8192	1.4377e-10	2.805e-09	1.1143e-12	1.6996e-15	1.2556e-14	1.28e-10
16384	9.0752e-10	5.9164e-09	3.1659e-12	6.084e-16	2.5771e-14	1.4976e-10
32768	2.1693e-09	7.3953e-11	2.6425e-12	5.99e-17	2.715e-14	1.9033e-10
\end{filecontents*}
\pgfplotstabletypeset[
		columns = {0,1,2,3,4,5,6},
		columns/0/.style = {column name = Size},
		columns/1/.style = {column name = A},
		columns/2/.style = {column name = B},
		columns/3/.style = {column name = C},
		columns/4/.style = {column name = D},
		columns/5/.style = {column name = E},
		columns/6/.style = {column name = F},
		]{toeplitz_residues.dat}
	}
	\caption{Relative residuals for \texttt{toeplitz\_solve} with global tolerance $\epsilon = 10^{-10}$ applied to the Toeplitz matrices A to F from~\cite{Xia2012}.}
	\label{tab:res1}
\end{table}
\subsection{Matrix functions for banded matrices}
\label{sec:bandedmatfun}

The computation of matrix functions arises in a variety of settings.
When $A$ is banded, the banded structure is sometimes numerically preserved by $f(A)$
\cite{Benzi2007,Benzi2013}, in the sense that $f(A)$ can be well approximated by a banded matrix. For example, this is the case for an entire function $f$ of a symmetric matrix $A$, provided that the width of the spectrum of $A$ remains modest.
In other cases, such as for matrices arising from the discretization of unbounded
operators $f(A)$ may lose approximate sparsity. Nevertheless, as discussed in
\cite{Gavrilyuk2004}, and demonstrated in the following, $f(A)$ can be highly structured and admit an accurate HSS or HODLR approximation.

As an example, we consider the function $f(z) = e^{z}$, and the 1D discrete Laplacian
\begin{equation} \label{eq:laplace}
A =
-\frac{1}{h^2}\begin{bmatrix}
2 & -1 \\
-1 & 2 & \ddots \\
& \ddots & \ddots & -1 \\
& & -1 & 2 \\
\end{bmatrix} \in \mathbb{C}^{n \times n}, \qquad h=\frac 1{n-1}.
\end{equation}
The \texttt{expm} function included in the toolbox
computes the exponential of $A$ in the HSS and HODLR format via
a Pad\'e expansion of degree $[13/13]$ combined with scaling and squaring~\cite{Higham2009}. For relatively small sizes (up to 16384), we
compare the execution time with  the one of the \texttt{expm} function included in \matlab. We compute a reference solution from the spectral decomposition of $A$, which is known in closed form, and use it to check the relative accuracy (in the spectral norm) of Matlab's \texttt{expm} and the corresponding HODLR and HSS functions; see Table~\ref{tab:expm}. The left plot of Figure~\ref{fig:expm} shows that the break-even point for the matrix size, where exploiting structure becomes beneficial in terms of execution time, is around $8192$. For nonsymmetric matrices, this threshold reduces to around
$1000$. For example, computing the exponential of the stiffness matrix of the convection diffusion problem in \cite[Section 5.3]{kressner2017low}, for $n=1024$, requires a computational time of about $1$ second with all the three versions of \texttt{expm}. For $n=4096$ we measure a computational time of about $110$ seconds for Matlab's \texttt{expm} and of about $4.5$ seconds for the corresponding HODLR/HSS functions.  One can also observe, in Figure~\ref{fig:expm}, the slightly better asymptotic complexity of HSS with respect to HODLR.

As the norm of $A$ grows
as $\mathcal O(n^2)$, the decay of off-diagonal entries can be expected to stay moderate. To verify this, we have computed a
sparse approximant to $e^{A}$ by discarding all entries smaller
than $10^{-5} \cdot \max_{i,j} |(e^{A})_{ij}|$ in the result obtained
with \matlab's \texttt{expm}. The threshold has been chosen a posteriori to ensure an accuracy similar to the one obtained with HODLR/HSS arithmetic. The right plot of Figure~\ref{fig:expm} shows that approximate sparsity is not very effective in this setting; the memory consumption still grows quadratically with $n$.
In contrast, the growth is much slower for the HODLR and HSS formats.

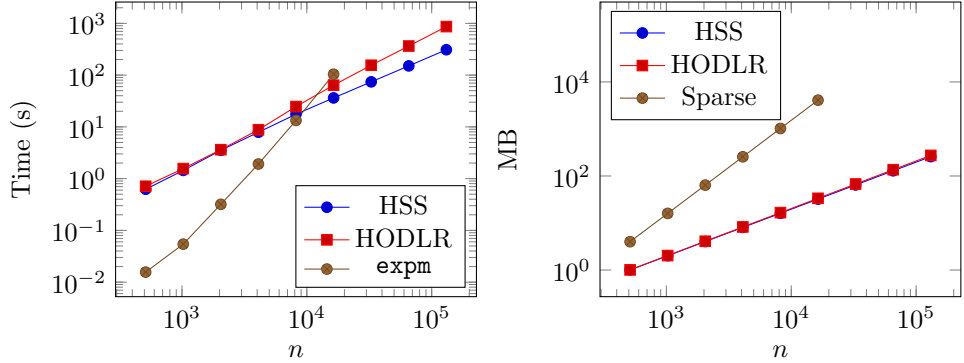
\begin{figure}
	\centering
	\begin{minipage}{.49\linewidth}
		\begin{tikzpicture}
		\begin{loglogaxis}[
		xlabel = $n$, ylabel = Time (s), 
		width=\linewidth, legend pos = south east]
		\addplot table[x index = 0, y index = 1] {expm.dat};
		\addplot table[x index = 0, y index = 2] {expm.dat};
		\addplot table[x index = 0, y index = 3] {expm.dat};
		\legend{HSS, HODLR, \texttt{expm}};
		\end{loglogaxis}
		\end{tikzpicture}    
	\end{minipage}~\begin{minipage}{.49\linewidth}
		\begin{tikzpicture}
		\begin{loglogaxis}[
		width=\linewidth, legend pos = north west, ymax = 5e5, ylabel = MB, xlabel = $n$]
		\addplot table[x index = 0, y index = 8] {expm.dat};
		\addplot table[x index = 0, y index = 9] {expm.dat};
		\addplot table[x index = 0, y index = 10] {expm.dat};
		\legend{HSS, HODLR, Sparse};
		\end{loglogaxis}
		\end{tikzpicture}    
	\end{minipage}
	\caption{Left: Execution times for computing of $e^A$, with $A$ being the discrete 1D Laplacian. Right:
		Memory consumption (in MBytes) in the HODLR and HSS formats, compared to the sparse
		approximant obtained by thresholding entries.}
	\label{fig:expm}
\end{figure}

\begin{table}
	\centering
	\begin{filecontents*}{expm.dat}
512	0.62305	0.71307	0.015627	4.2938e-09	4.1195e-09	6.5647e-11	6.5647e-11	1.0079	1.0079	4.0039	1.0445e+06
1024	1.4434	1.5631	0.054248	1.7421e-08	1.791e-08	2.8604e-10	3.666e-08	2.0159	2.0314	16.008	4.1861e+06
2048	3.5296	3.6225	0.31937	7.3661e-08	7.2394e-08	1.472e-09	2.9719e-08	4.032	4.094	64.015	1.6761e+07
4096	7.8731	8.7989	1.9175	3.0768e-07	2.9733e-07	4.7415e-09	2.7246e-08	8.0641	8.2505	256.03	6.7076e+07
8192	17.552	24.688	13.336	1.1485e-06	1.1399e-06	1.8763e-08	2.7025e-08	16.128	16.626	1024	2.6837e+08
16384	36.265	63.863	103.97	4.8129e-06	4.6757e-06	6.5337e-08	6.5337e-08	32.257	33.502	4096	1.0736e+09
32768	74.122	155.13	0	0	0	0	0	64.514	67.504	0	4.2947e+09
65536	150.73	364.69	0	0	0	0	0	129.03	136.01	0	1.7179e+10
1.3107e+05	309.49	866.2	0	0	0	0	0	258.05	274.02	0	6.8718e+10
\end{filecontents*}
\pgfplotstabletypeset[
	skip rows between index={6}{14},
	columns = {0,4,5,6, 11},
	columns/0/.style = {column name = $n$},
	columns/4/.style = {column name = Error (HSS)},
	columns/5/.style = {column name = Error (HODLR)},
	columns/6/.style = {column name = Error (\texttt{expm})},
	columns/11/.style = {column name = $\norm{A}_2$}
	]{expm.dat}  
	\caption{Relative errors for the approximation of the matrix exponential in the HODLR and HSS format. }
	\label{tab:expm}
\end{table}

\subsection{Matrix equations and 2D fractional PDEs}
\label{sec:2dfractional}
It has been recently noticed that discretizations of 1D fractional differential operators $\frac{\partial^{\alpha}}{\partial x^{\alpha}}$,  $\alpha\in(1,2)$, can be efficiently represented by HODLR matrices \cite{massei2018fast}. We consider 2D separable operators arising from a fractional PDE of the form 
\begin{equation}\label{eq:L-op}
\begin{cases}\frac{\partial^{\alpha}u(x,y)}{\partial x^\alpha}+\frac{\partial^{\alpha}u(x,y)}{\partial y^\alpha}=f(x,y)&(x,y)\in\Omega:=(0,1)^2\\
u(x,y)\equiv 0& (x,y)\in \mathbb R^2\setminus\Omega
\end{cases}.
\end{equation}
Discretizing \eqref{eq:L-op} on a tensorized $(n+2)\times (n+2)$ grid provides an $n^2\times n^2$ matrix of the form $\mathcal M:=A\otimes I+I\otimes A$ and a vector $b\in\mathbb R^{n^2}$ containing the representation of  the right hand side $f(x,y)$. Thanks to the Kronecker structure, the linear system $\mathcal Mx=b$ can be recast into the matrix equation
\begin{equation}\label{eq:lyap}
AX+XA^T=C,\qquad \mathrm{vec}(C)=b,\qquad \mathrm{vec}(X)=x.
\end{equation}
If $C$ is a low-rank matrix --- a condition sometimes satisfied in the applications --- the solution $X$ is numerically low-rank and it is efficiently approximated via \emph{rational Krylov subspace methods} \cite{Simoncini2016}. The latter require fast procedures for the matrix-vector product and the solution of shifted linear systems with the matrix $A$. If $A$ is represented in the HODLR or HSS format this requirement is satisfied. In particular,
the Lyapunov solver \texttt{ek\_lyap} included the \hm{} is based on the \emph{extended Krylov subspace method}, described in \cite{Simoncini2016}.

We consider a simple example where we choose $\alpha =1.7$ and the finite difference discretization described in \cite{Meerschaert}.  In this setting, the matrix $A$ is 
given by
\[
A=T_{\alpha,n}+T_{\alpha,n}^T,\qquad 	T_{\alpha,n}=-\frac{1}{\Delta x^{\alpha}}\left[
\begin{matrix}
g_1^{(\alpha)}& g_0^{(\alpha)} & 0 & \cdots & 0 &0\\
g_2^{(\alpha)}& g_1^{(\alpha)} & g_0^{(\alpha)}& 0 &\cdots & 0\\
\vdots &\ddots & \ddots & \ddots &\ddots & \vdots\\
\vdots &\ddots & \ddots & \ddots &\ddots & 0\\
g_{n-1}^{(\alpha)}&\ddots & \ddots & \ddots & g_1^{(\alpha)}& g_0^{(\alpha)}\\
g_{n}^{(\alpha)}& g_{n-1}^{(\alpha)} & \cdots & \cdots & g_2^{(\alpha)}& g_1^{(\alpha)}
\end{matrix}
\right],
\]
where $g_j^{(\alpha)}=(-1)^j
\binom{\alpha}{k}$. The matrix $A$ has Toeplitz structure and it has been proven to have off-diagonal blocks of (approximate) low rank in \cite{massei2018fast}. The source term is  $f(x,y)=\sin(2\pi x)\sin(2\pi y)$
and the matrix $C$ containing its samplings  has rank $1$.

To retrieve the HODLR representation of $A$ we rely on the Toeplitz constructor:
\begin{lstlisting}
Dx = 1/(n + 2);
[c, r] = fractional_symbol(alpha, n);
T = hodlr('toeplitz', c, r, n) / Dx^alpha;
A = T + T';
\end{lstlisting}
We combine this with \texttt{ek\_lyap} in order to solve~\eqref{eq:lyap}: 
\begin{lstlisting}
u = sin(2 * pi * (1:n) / (n + 2))';
Xu = ek_lyap(A, u, inf, 1e-6);
X = hodlr('low-rank', Xu, Xu);
U = hodlr('low-rank', u, u);
Res = norm(A * X + X * A + U) / norm(U);
\end{lstlisting}

\begin{table}
	\begin{filecontents*}{frac.dat}
1024	0.064463	0.072674	0.13714	1.2113e-08	17	0.14078	0.19373	0.33452	1.2113e-08	16	23
2048	0.063879	0.13766	0.20154	1.1936e-08	18	0.2383	0.41821	0.65651	1.1936e-08	18	27
4096	0.13593	0.39966	0.53559	9.0268e-09	19	0.47411	0.97957	1.4537	9.0255e-09	19	31
8192	0.30164	0.98237	1.284	9.9523e-09	20	1.0278	2.1766	3.2044	9.9515e-09	19	35
16384	0.64867	2.3377	2.9864	8.1718e-09	21	1.9801	4.4385	6.4186	8.3977e-09	21	39
32768	1.3161	5.365	6.6811	1.1524e-08	22	4.1275	8.8491	12.977	8.8156e-09	22	42
65536	2.8271	12.078	14.905	1.0817e-08	23	8.1222	18.938	27.06	9.8316e-09	21	46
1.3107e+05	5.7135	26.989	32.703	5.4967e-08	23	16.891	43.432	60.323	2.7401e-08	21	50
\end{filecontents*}
\pgfplotstabletypeset[
	every head row/.style={
		before row={
			\toprule
			& \multicolumn{3}{c|}{\lstinline{hodlr}}
			&\multicolumn{3}{c|}{\lstinline{hss}}\\
		},
		after row = \midrule,
	},
	columns = {0,1,3,4,6,8,9,11},
	columns/0/.style = {column name = Size, column type=c|},
	columns/1/.style = {column name = $T_{\mathrm{build}}$, fixed},
	columns/3/.style = {column name = $T_{\mathrm{tot}}$,fixed},
	columns/4/.style = {column name = Res, column type=c|},
	columns/6/.style = {column name = $T_{\mathrm{build}}$, fixed},
	columns/8/.style = {column name = $T_{\mathrm{tot}}$},
	columns/9/.style = {column name = {Res}},
	columns/11/.style = {column name = $\mathrm{rank}(X)$, column type=|c}
	]{frac.dat}
	\caption{Performances of \texttt{ek\_lyap} with HODLR and HSS matrices}
	\label{tab:frac}
\end{table}
The obtained results are reported in Table~\ref{tab:frac} where
\begin{itemize}
	\item $T_{\mathrm{build}}$ indicates the time for constructing the HODLR or HSS representation,
	\item $T_{\mathrm{tot}}$ indicates the total time of the procedure,
	\item $\mathrm{Res}$ denotes the residual associated with the approximate solution $X$: $\norm{AX+XA+C}_2/\norm{C}_2$.
\end{itemize}
The results demonstrate the linear poly-logarithmic asymptotic complexity
of the proposed scheme.

\section{Conclusions}

We have presented the \hm{}, a Matlab software for working with HODLR and HSS matrices. Based on state-of-the-art and newly developed algorithms, its functionality matches much of the functionality available in Matlab for dense matrices, while most existing software packages for matrices with hierarchical low-rank structures focus on specific tasks, most notably linear systems. Nevertheless, there is room for further improvement and future work. In particular, the range of constructors could be extended further by advanced techniques based on function expansions and randomized sampling.
Also, the full range of matrix functions and other non-standard linear algebra tasks is not fully exhausted by our toolbox.

\appendix
\section{HSS re-compression: algorithm and error analysis} \label{app:A}

Here, we provide a description and an analysis of the algorithm from~\cite[Section 5]{Xia2010}, which performs the recompression of an HSS matrix $A$ with respect to a certain tolerance $\tau$. As discussed in Section~\ref{sec:compression}, we suppose that  $A$ is already in proper form, i.e., its factors $U_i^{(\ell)},V_i^{(\ell)}$ have orthonormal columns for all $i,\ell$.

The recompression procedures handles HSS block rows and HSS block columns in an analogous manner; to simplify the exposition we only describe the compression of HSS block rows. For this purpose, we consider the following partition of the translation operators
\[
R_{U,i}^{(\ell)}=\begin{bmatrix}
R_{U,i,1}^{(\ell)}\\
R_{U,i,2}^{(\ell)}
\end{bmatrix}\in\mathbb C^{2k\times k},\qquad R_{U,i,h}^{(\ell)}\in\mathbb C^{k\times k}\ \ h=1,2.
\]
For each level $\ell=1,\dots,p$ and every $i=1,\dots,2^\ell$ the algorithm has access to a matrix $W_i$ --- having $k$ rows --- such that the $i$th HSS block row can be written as
\begin{equation}\label{eq:hss-block}
\begin{bmatrix}
U_{2i-1}^{(\ell+1)}R_{U,i,1}^{(\ell)}W_i\widetilde V_i^*\\
U_{2i}^{(\ell+1)}R_{U,i,2}^{(\ell)}W_i\widetilde V_i^*
\end{bmatrix}
\end{equation}
for some matrix $\widetilde V_i$ having orthonormal columns. At level $\ell=1$ the algorithm chooses $W_1=S_{1,2}^{(1)}$, $W_2=S_{2,1}^{(1)}$, $\widetilde V_1=V_2^{(1)}$, and $\widetilde V_2=V_1^{(1)}$. Note that the relation~\eqref{eq:hss-block} allows us to write the HSS block rows at level $\ell+1$ as 
\[
U_{2i-1}^{(\ell+1)}\begin{bmatrix}
S_{i,i+1}^{(\ell+1)}&R_{U,i,1}^{(\ell)}W_i
\end{bmatrix}\breve{V}_1^*,\qquad U_{2i}^{(\ell+1)}\begin{bmatrix}
S_{i+1,i}^{(\ell+1)}&R_{U,i,2}^{(\ell)}W_i
\end{bmatrix}\breve{V}_2^*,
\] 
where $\breve{V}_1,\breve{V}_2$ are suitable row permutations of $\widetilde V_i\oplus V_{2i}^{(\ell)} $ and $\widetilde V_i\oplus V_{2i-1}^{(\ell)} $, respectively.

The algorithm proceeds with the following steps:
\begin{itemize}
	\item compute the truncated SVDs (neglecting singular values below the tolerance $\tau$)  \begin{align*}\widehat U_1\widehat S_1\begin{bmatrix}\widehat V_{11}^*&\widehat V_{12}^*\end{bmatrix}&\approx \begin{bmatrix}
	S_{i,i+1}^{(\ell)}& R_{U,i,1}^{(\ell)} W_i\end{bmatrix},\\ \widehat U_2\widehat S_2\begin{bmatrix}\widehat V_{21}^*&\widehat V_{22}^*\end{bmatrix}&\approx \begin{bmatrix}S_{i+1,i}^{(\ell)}& R_{U,i,2}^{(\ell)} W_i\end{bmatrix}.\end{align*} 
	In particular, we have the approximate factorizations
	\begin{align*}
	U_{2i-1}^{(\ell+1)}\begin{bmatrix}
	S_{i,i+1}^{(\ell+1)}&R_{U,i,1}^{(\ell)}W_i
	\end{bmatrix}\breve{V}_1^*\approx U_{2i-1}^{(\ell+1)}\widehat U_1\begin{bmatrix}
	\widehat S_1\widehat V_{11}^*&\widehat U_1^*R_{U,i,1}^{(\ell)}W_i
	\end{bmatrix}\breve{V}_1^*,\\
	U_{2i}^{(\ell+1)}\begin{bmatrix}
	S_{i+1,i}^{(\ell+1)}&R_{U,i,2}^{(\ell)}W_i
	\end{bmatrix}\breve{V}_2^*\approx U_{2i}^{(\ell+1)}\widehat U_2\begin{bmatrix}
	\widehat S_2\widehat V_{21}^*&\widehat U_2^*R_{U,i,2}^{(\ell)}W_i
	\end{bmatrix}\breve{V}_2^*.
	\end{align*}	
	\item The above factorizations are equivalent to performing the following updates \begin{align*} S_{i,i+1}^{(\ell)}&=\widehat S\widehat V_{11}^*,& R_{U,2i-1}^{(\ell+1)}&=R_{U,2i-1}^{(\ell+1)}\widehat U_1,
	&R_{U,i}^{(\ell)}&=\begin{bmatrix}
	\widehat U_1\\ & \widehat U_2
	\end{bmatrix}R_{U,i}^{(\ell)}, \\
	S_{i+1,i}^{(\ell)}&=\widehat S\widehat V_{21}^*,& R_{U,2i}^{(\ell+1)}&=R_{U,2i}^{(\ell+1)}\widehat U_2.
	\end{align*}
\end{itemize}
The analogous operations are performed on the HSS block columns. 
We notice that the truncated SVDs introduced an error with norm bounded by $\tau$ in every HSS block row and column on every level. 
This leads to the following. 
\begin{proposition}
	Let $A$ be a $(\mathcal T_p,k)$-HSS matrix for some $p,k\in\mathbb N$ and $\widetilde A$ the output of the recompression algorithm described above, using the truncation tolerance $\tau>0$. Then, $\norm{A-\widetilde A}_2\leq 2\frac{\sqrt{2^p}-1}{\sqrt 2-1}\tau$.
\end{proposition}
\begin{proof}
	We remark that at each level $\ell$ the algorithm introduces a row and column perturbations of the form
	\[
	E^{(\ell)}+(F^{(\ell)})^T=
	\begin{bmatrix}
	E_1^{(\ell)}&\dots&	E_{2^\ell}^{(\ell)}
	\end{bmatrix}+
	\begin{bmatrix}
	F_1^{(\ell)}&\dots&	F_{2^\ell}^{(\ell)}
	\end{bmatrix}^T
	\]
	where $E_j^{(\ell)},F_j^{(\ell)}$ have norm bounded by $\tau$ for every $j$. Since $\norm{E^{(\ell)}}_2,\norm{F^{(\ell)}}_2\leq \sqrt {2^\ell}\tau$, the claim follows by summing for $\ell=1,\dots,p$.
\end{proof}

\bibliography{hm,anchp}

\begin{thebibliography}{10}

\bibitem{Ambikasaran2016}
S.~Ambikasaran, D.~Foreman-Mackey, L.~Greengard, D.~W. Hogg, and M.~O’Neil.
\newblock Fast direct methods for {G}aussian processes.
\newblock {\em IEEE Trans. Pattern Anal. Mach. Intell.}, 38(2):252--265, 2016.

\bibitem{Ambikasaran2019HODLRlib}
S.~Ambikasaran, K.~Singh, and S.~Sankaran.
\newblock {HODLRlib}: A {L}ibrary for {H}ierarchical {M}atrices.
\newblock {\em Journal of Open Source Software}, 4(34):1167, 2 2019.

\bibitem{Ballani2016}
J.~Ballani and D.~Kressner.
\newblock Matrices with hierarchical low-rank structures.
\newblock In {\em Exploiting hidden structure in matrix computations:
  algorithms and applications}, volume 2173 of {\em Lecture Notes in Math.},
  pages 161--209. Springer, Cham, 2016.

\bibitem{AHMED}
M.~Bebendorf.
\newblock {AHMED}.
\newblock https://github.com/xantares/ahmed, September 2019.

\bibitem{BebendorfH2003}
M.~Bebendorf and W.~Hackbusch.
\newblock Existence of {$\mathcal H$}-matrix approximants to the inverse
  {FE}-matrix of elliptic operators with {$L^\infty$}-coefficients.
\newblock {\em Numer. Math.}, 95(1):1--28, 2003.

\bibitem{Benzi2013}
M.~Benzi, P.~Boito, and N.~Razouk.
\newblock Decay properties of spectral projectors with applications to
  electronic structure.
\newblock {\em SIAM Rev.}, 55(1):3--64, 2013.

\bibitem{Benzi2007}
M.~Benzi and N.~Razouk.
\newblock Decay bounds and {$O(n)$} algorithms for approximating functions of
  sparse matrices.
\newblock {\em Electron. Trans. Numer. Anal.}, 28:16--39, 2007.

\bibitem{Bini2017cyclic}
D.~A. Bini, S.~Massei, and L.~Robol.
\newblock Efficient cyclic reduction for quasi-birth-death problems with rank
  structured blocks.
\newblock {\em Appl. Numer. Math.}, 116:37--46, 2017.

\bibitem{Bini2017}
D.~A. Bini, S.~Massei, and L.~Robol.
\newblock On the decay of the off-diagonal singular values in cyclic reduction.
\newblock {\em Linear Algebra Appl.}, 519:27--53, 2017.

\bibitem{HLIBpro}
S.~B\"orm.
\newblock {HLIBPro}.
\newblock https://www.hlibpro.com/.

\bibitem{Borm2006}
S.~{B{\"o}rm}.
\newblock {$\mathcal{H}_2$}-matrices -- an efficient tool for the treatment of
  dense matrices.
\newblock Habilitationsschrift, Christian-Albrechts-Universit\"at zu Kiel,
  2006.

\bibitem{Borm2010}
S.~B{\"o}rm.
\newblock {\em Efficient numerical methods for non-local operators}, volume~14
  of {\em EMS Tracts in Mathematics}.
\newblock European Mathematical Society (EMS), Z\"urich, 2010.
\newblock ${\mathcal{H}}{^{2}}$-matrix compression, algorithms and analysis.

\bibitem{H2Lib}
S.~B\"orm.
\newblock {H2Lib}.
\newblock https://github.com/H2Lib/H2Lib, September 2019.

\bibitem{Borm2003}
S.~B\"orm, L.~Grasedyck, and W.~Hackbusch.
\newblock Hierarchical matrices.
\newblock Lecture note 21/2003, MPI-MIS Leipzig, Germany, 2003.
\newblock Revised June 2006. Available from
  \url{http://www.mis.mpg.de/preprints/ln/lecturenote-2103.pdf}.

\bibitem{Businger}
P.~Businger and G.~H. Golub.
\newblock Handbook series linear algebra. {L}inear least squares solutions by
  {H}ouseholder transformations.
\newblock {\em Numer. Math.}, 7:269--276, 1965.

\bibitem{Chandrasekaran2005}
S.~Chandrasekaran, P.~Dewilde, M.~Gu, T.~Pals, X.~Sun, A.-J. van~der Veen, and
  D.~White.
\newblock Some fast algorithms for sequentially semiseparable representations.
\newblock {\em SIAM J. Matrix Anal. Appl.}, 27(2):341--364, 2005.

\bibitem{Chandrasekaran2006}
S.~Chandrasekaran, M.~Gu, and T.~Pals.
\newblock A fast {ULV} decomposition solver for hierarchically semiseparable
  representations.
\newblock {\em SIAM J. Matrix Anal. Appl.}, 28(3):603--622, 2006.

\bibitem{Doelz2019}
J.~D\"{o}lz, H.~Harbrecht, and M.~D. Multerer.
\newblock On the best approximation of the hierarchical matrix product.
\newblock {\em SIAM J. Matrix Anal. Appl.}, 40(1):147--174, 2019.

\bibitem{Eidelman2014a}
Y.~Eidelman, I.~Gohberg, and I.~Haimovici.
\newblock {\em Separable type representations of matrices and fast algorithms.
  {V}ol. 1}, volume 234 of {\em Operator Theory: Advances and Applications}.
\newblock Birkh\"{a}user/Springer, Basel, 2014.
\newblock Basics. Completion problems. Multiplication and inversion algorithms.

\bibitem{Eidelman2014b}
Y.~Eidelman, I.~Gohberg, and I.~Haimovici.
\newblock {\em Separable type representations of matrices and fast algorithms.
  {V}ol. 2}, volume 235 of {\em Operator Theory: Advances and Applications}.
\newblock Birkh\"{a}user/Springer Basel AG, Basel, 2014.
\newblock Eigenvalue method.

\bibitem{Faustmann2017}
M.~Faustmann, J.~M. Melenk, and D.~Praetorius.
\newblock Existence of {$\mathcal{H}$}-matrix approximants to the inverse of
  {BEM} matrices: the hyper-singular integral operator.
\newblock {\em IMA J. Numer. Anal.}, 37(3):1211--1244, 2017.

\bibitem{Gavrilyuk}
I.~P. Gavrilyuk, W.~Hackbusch, and B.~N. Khoromskij.
\newblock {$\mathcal{H}$}-matrix approximation for the operator exponential
  with applications.
\newblock {\em Numer. Math.}, 92(1):83--111, 2002.

\bibitem{Gavrilyuk2004}
I.~P. Gavrilyuk, W.~Hackbusch, and B.~N. Khoromskij.
\newblock Data-sparse approximation to the operator-valued functions of
  elliptic operator.
\newblock {\em Math. Comp.}, 73(247):1297--1324, 2004.

\bibitem{Geoga2019}
C.~J. Geoga, M.~Anitescu, and M.~L. Stein.
\newblock Scalable {G}aussian process computations using hierarchical matrices.
\newblock {\em Journal of Computational and Graphical Statistics}, 0(ja):1--22,
  2019.

\bibitem{Ghysels2016}
P.~Ghysels, X.~S. Li, F.-H. Rouet, S.~Williams, and A.~Napov.
\newblock An efficient multicore implementation of a novel {HSS}-structured
  multifrontal solver using randomized sampling.
\newblock {\em SIAM J. Sci. Comput.}, 38(5):S358--S384, 2016.

\bibitem{Golub2013}
G.~H. Golub and C.~F. Van~Loan.
\newblock {\em Matrix computations}.
\newblock Johns Hopkins Studies in the Mathematical Sciences. Johns Hopkins
  University Press, Baltimore, MD, fourth edition, 2013.

\bibitem{Grasedyck2004a}
L.~Grasedyck.
\newblock Existence of a low rank or {$\mathcal H$}-matrix approximant to the
  solution of a {S}ylvester equation.
\newblock {\em Numer. Linear Algebra Appl.}, 11(4):371--389, 2004.

\bibitem{Hackbusch2015}
W.~Hackbusch.
\newblock {\em Hierarchical matrices: algorithms and analysis}, volume~49 of
  {\em Springer Series in Computational Mathematics}.
\newblock Springer, Heidelberg, 2015.

\bibitem{Hackbusch2004}
W.~Hackbusch, B.~N. Khoromskij, and R.~Kriemann.
\newblock Hierarchical matrices based on a weak admissibility criterion.
\newblock {\em Computing}, 73(3):207--243, 2004.

\bibitem{Halko2011}
N.~Halko, P.~G. Martinsson, and J.~A. Tropp.
\newblock Finding structure with randomness: probabilistic algorithms for
  constructing approximate matrix decompositions.
\newblock {\em SIAM Rev.}, 53(2):217--288, 2011.

\bibitem{Heinig1995}
G.~Heinig.
\newblock Inversion of generalized {C}auchy matrices and other classes of
  structured matrices.
\newblock In {\em Linear algebra for signal processing ({M}inneapolis, {MN},
  1992)}, volume~69 of {\em IMA Vol. Math. Appl.}, pages 63--81. Springer, New
  York, 1995.

\bibitem{Higham2008}
N.~J. Higham.
\newblock {\em Functions of matrices}.
\newblock SIAM, Philadelphia, PA, 2008.

\bibitem{Higham2009}
N.~J. Higham.
\newblock The scaling and squaring method for the matrix exponential revisited.
\newblock {\em SIAM Rev.}, 51(4):747--764, 2009.

\bibitem{KressnerKM2019}
D.~Kressner, P.~K\"urschner, and S.~Massei.
\newblock Low-rank updates and divide-and-conquer methods for quadratic matrix
  equations.
\newblock {arXiv}:1903.02343, 2019.
\newblock To appear in Numer. Algorithms.

\bibitem{kressner2017low}
D.~Kressner, S.~Massei, and L.~Robol.
\newblock Low-rank updates and a divide-and-conquer method for linear matrix
  equations.
\newblock {\em SIAM J. Sci. Comput.}, 41(2):A848--A876, 2019.

\bibitem{kressner2017recompression}
D.~Kressner and L.~Peri\v{s}a.
\newblock Recompression of {H}adamard products of tensors in {T}ucker format.
\newblock {\em SIAM J. Sci. Comput.}, 39(5):A1879--A1902, 2017.

\bibitem{KressnerS2017}
D.~Kressner and A.~\v{S}u\v{s}njara.
\newblock Fast computation of spectral projectors of banded matrices.
\newblock {\em SIAM J. Matrix Anal. Appl.}, 38(3):984--1009, 2017.

\bibitem{kressner2018fast}
D.~Kressner and A.~\v{S}u\v{s}njara.
\newblock Fast {QR} decomposition of {HODLR} matrices.
\newblock {\em arXiv preprint arXiv:1809.10585}, 2018.

\bibitem{Martinsson2011}
P.~G. Martinsson.
\newblock A fast randomized algorithm for computing a hierarchically
  semiseparable representation of a matrix.
\newblock {\em SIAM J. Matrix Anal. Appl.}, 32(4):1251--1274, 2011.

\bibitem{massei2018fast}
S.~Massei, M.~Mazza, and L.~Robol.
\newblock Fast solvers for two-dimensional fractional diffusion equations using
  rank structured matrices.
\newblock {\em SIAM J. Sci. Comput.}, 41(4):A2627--A2656, 2019.

\bibitem{Massei2017quasi}
S.~Massei and L.~Robol.
\newblock Decay bounds for the numerical quasiseparable preservation in matrix
  functions.
\newblock {\em Linear Algebra Appl.}, 516:212--242, 2017.

\bibitem{Meerschaert}
M.~M. Meerschaert and C.~Tadjeran.
\newblock Finite difference approximations for fractional advection-dispersion
  flow equations.
\newblock {\em J. Comput. Appl. Math.}, 172(1):65--77, 2004.

\bibitem{Rouet2016}
F.~Rouet, X.~S. Li, P.~Ghysels, and A.~Napov.
\newblock A distributed-memory package for dense hierarchically semi-separable
  matrix computations using randomization.
\newblock {\em ACM Trans. Math. Software}, 42(4):Art. 27, 35, 2016.

\bibitem{sheng}
Z.~Sheng, P.~Dewilde, and S.~Chandrasekaran.
\newblock Algorithms to solve hierarchically semi-separable systems.
\newblock In {\em System theory, the {S}chur algorithm and multidimensional
  analysis}, volume 176 of {\em Oper. Theory Adv. Appl.}, pages 255--294.
  Birkh\"{a}user, Basel, 2007.

\bibitem{Simon2000}
H.~D. Simon and H.~Zha.
\newblock Low-rank matrix approximation using the {L}anczos bidiagonalization
  process with applications.
\newblock {\em SIAM J. Sci. Comput.}, 21(6):2257--2274, 2000.

\bibitem{Simoncini2016}
V.~Simoncini.
\newblock Computational methods for linear matrix equations.
\newblock {\em SIAM Rev.}, 58(3):377--441, 2016.

\bibitem{Vandebril2008}
R.~Vandebril, M.~Van~Barel, and N.~Mastronardi.
\newblock {\em Matrix computations and semiseparable matrices. {V}ol. 1}.
\newblock Johns Hopkins University Press, Baltimore, MD, 2008.
\newblock Linear systems.

\bibitem{Vandebril2008a}
R.~Vandebril, M.~Van~Barel, and N.~Mastronardi.
\newblock {\em Matrix computations and semiseparable matrices. {V}ol. 2}.
\newblock Johns Hopkins University Press, Baltimore, MD, 2008.
\newblock Eigenvalue and singular value methods.

\bibitem{Vogel2016}
J.~Vogel, J.~Xia, S.~Cauley, and V.~Balakrishnan.
\newblock Superfast divide-and-conquer method and perturbation analysis for
  structured eigenvalue solutions.
\newblock {\em SIAM J. Sci. Comput.}, 38(3):A1358--A1382, 2016.

\bibitem{SusnjaraK2018}
A.~\v{S}u\v{s}njara and D.~Kressner.
\newblock A fast spectral divide-and-conquer method for banded matrices.
\newblock {arXiv}:1801.04175, 2018.

\bibitem{Wang2016}
S.~Wang, X.~S. Li, F.-H. Rouet, J.~Xia, and M.~V. de~Hoop.
\newblock A parallel geometric multifrontal solver using hierarchically
  semiseparable structure.
\newblock {\em ACM Trans. Math. Software}, 42(3):Art. 21, 21, 2016.

\bibitem{Xi2014}
Y.~Xi, J.~Xia, S.~Cauley, and V.~Balakrishnan.
\newblock Superfast and stable structured solvers for {T}oeplitz least squares
  via randomized sampling.
\newblock {\em SIAM J. Matrix Anal. Appl.}, 35(1):44--72, 2014.

\bibitem{Xia2009}
J.~Xia, S.~Chandrasekaran, M.~Gu, and X.~S. Li.
\newblock Superfast multifrontal method for large structured linear systems of
  equations.
\newblock {\em SIAM J. Matrix Anal. Appl.}, 31(3):1382--1411, 2009.

\bibitem{Xia2010}
J.~Xia, S.~Chandrasekaran, M.~Gu, and X.~S. Li.
\newblock Fast algorithms for hierarchically semiseparable matrices.
\newblock {\em Numer. Linear Algebra Appl.}, 17(6):953--976, 2010.

\bibitem{Xia2012}
J.~Xia, Y.~Xi, and M.~Gu.
\newblock A superfast structured solver for {T}oeplitz linear systems via
  randomized sampling.
\newblock {\em SIAM J. Matrix Anal. Appl.}, 33(3):837--858, 2012.

\end{thebibliography}
\bibliographystyle{plain}

\end{document}